\documentclass[final,10pt]{amsart}
\usepackage[left=1.4in, right=1.4in, top=1.5in, bottom=1.5in]{geometry}

\usepackage[foot]{amsaddr} %package that puts the authors addresses on the title page as a footnote (when using AMSART document class only).

\usepackage[]{showkeys}

\usepackage{verbatim}
\usepackage{amssymb}
\usepackage{amsthm}
\usepackage{amsmath}
\usepackage{amsfonts}
\usepackage{latexsym}
\usepackage{mathtools} 
\usepackage{enumitem} 
\usepackage[english]{babel} 
\usepackage{mathabx}

\usepackage{float}

\usepackage[usenames,dvipsnames]{xcolor} 
\definecolor{dkblue}{RGB}{30,90,140} % This is a dark Blue     
\usepackage[colorlinks=true, pdfstartview=FitV, linkcolor=dkblue, citecolor=dkblue, urlcolor=dkblue]{hyperref}    

\definecolor{mydarkbluett}{RGB}{12,111,174}

\usepackage{todonotes}
\emergencystretch=\maxdimen
\theoremstyle{plain}
\newtheorem{thm}{Theorem}
\newtheorem{remark}[thm]{Remark}

\newtheorem{lemma}[thm]{Lemma}

\numberwithin{equation}{section}
\numberwithin{thm}{section}

\newcommand{\pa}{\partial_{\alpha}}

\newcommand{\minspace}{\hspace{0.05cm}}
\newcommand{\mminspace}{\hspace{0.02cm}}

\newcommand{\R}{\mathbb R}

\newcommand{\mumin}{\mu^{m}}
\newcommand{\mumax}{\mu^{M}}
\newcommand{\muin}{\mu^{in}}
\newcommand{\muout}{\mu^{out}}

\newcommand{\rhomin}{\rho^{m}}
\newcommand{\mum}{\overline{\mu}}

\newcommand{\intrtwo}{\int_{\mathbb{R}^2}}
\newcommand{\pai}{\partial_i}
\newcommand{\pj}{\partial_j}
\newcommand{\pk}{\partial_k}
\newcommand{\pam}{\partial_m}
\newcommand{\pone}{\partial_1}
\newcommand{\ptwo}{\partial_2}

\DeclareMathOperator{\Lip}{Lip}

\begin{document}

\title[Global regularity of 2D N-S free boundary with small viscosity contrast]{Global regularity of 2D Navier-Stokes free\\ boundary with small viscosity contrast}

\author[F. Gancedo]{Francisco Gancedo$^\dagger$}
\address{$^\dagger$Departamento de An\'{a}lisis Matem\'{a}tico $\&$ IMUS, Universidad de Sevilla, C/ Tarfia s/n, Campus Reina Mercedes, 41012 Sevilla, Spain.  \href{mailto:fgancedo@us.es}{fgancedo@us.es}}

% \address{$^\ddagger$Department of Mathematics, University of Pennsylvania, Philadelphia, PA 19104, USA.  \href{mailto:edugar@math.upenn.edu}{edugar@math.upenn.edu}}
\author[E. Garc\'ia-Ju\'arez]{Eduardo Garc\'ia-Ju\'arez$^{\ddagger}$}
\address{$^{\ddagger}$Departament de Matemàtiques i Informàtica, Universitat de Barcelona, Gran Via de les Corts Catalanes, 585 08007, Barcelona, Spain. \href{mailto:egarciajuarez@ub.edu}{egarciajuarez@ub.edu}}

%\date{\today}
%\date{\today; \Red{(DRAFT)}}
% \date{May 8, 2019}
% \date{}

\begin{abstract}
This paper studies the dynamics of two incompressible immiscible fluids in 2D modeled by the inhomogeneous Navier-Stokes equations. We prove that if initially the viscosity contrast is small then there is global-in-time regularity. This result has been proved recently in \cite{PaicuZhang2020} for $H^{5/2}$ Sobolev regularity  of the interface. Here we provide a new approach which allows to obtain preservation of the natural $C^{1+\gamma}$ H\"older regularity of the interface for all $0<\gamma<1$. Our proof is direct and allows for low Sobolev regularity of the initial velocity without any extra technicality. It uses new quantitative harmonic analysis bounds for $C^{\gamma}$ norms of even singular integral operators on characteristic functions of $C^{1+\gamma}$ domains \cite{GancedoG-J2021}.

\end{abstract}

% set the depth for the table of contents (0-2)
\setcounter{tocdepth}{1}
%\chapter is level 0
%\section is level 1
%\subsection is level 2
%\subsubsection is level 3
%\paragraph is level 4
%\subparagraph is level 5

\maketitle
%\tableofcontents

\section{Introduction}

In this paper we consider incompressible flows in the whole space $\R^2$,
\begin{equation}\label{incom}
\nabla\cdot u=0,
\end{equation}
of inhomogeneous fluids
\begin{equation}\label{mass}
\partial_t\rho+u\cdot\nabla \rho=0,   
\end{equation}
driven by Navier-Stokes equations
\begin{equation}
\label{NavierStokes}
\rho D_t u =\nabla\cdot\left(\mu\mathbb{D}u - \mathbb{I}_2P\right).
\end{equation}
Above $u$, $\rho$, $\mu$ and $P$ are the velocity field, the density, viscosity and pressure of the fluids. The operator  $D_t$ is the total derivative
\begin{equation*}
    D_t u=\partial_t u +u\cdot\nabla u,
\end{equation*}
the tensor $\mathbb{D}u$ denotes the symmetric part of the gradient
$$\mathbb{D}u=\nabla u +\nabla u^*,\quad \mathbb{D}_{ij}u=\pai u_j+\pj u_i,$$ and $\mathbb{I}_2$ is the identity matrix in $\mathbb{R}^2$. The viscosity depends smoothly on the density, $\mu=\tilde{\mu}(\rho)$ with $\tilde{\mu}$ smooth, so that in particular it is also preserved along trajectories
\begin{equation}\label{viscosity}
\partial_t\mu+u\cdot\nabla \mu=0.
\end{equation}
We deal with a moving fluid occupying a bounded domain $D(t)\subset\R^2$ and a second fluid occupying the complement of it $D(t)^c=\R^2\smallsetminus\overline{D(t)}$. They evolve with the velocity field through the particle trajectories
\begin{equation}\label{particle_trajectories}
\left\{\begin{aligned}
\frac{dX}{dt}(y,t)&=u(X(y,t),t),\\
X(y,0)&=y.
\end{aligned}\right.
\end{equation}
The fluids are immiscible, having different characteristics, principally different densities and viscosities, so that \begin{equation}\label{jump}
(u,\rho,\mu,P)(x,t)=\left\{\begin{array}{ll}
		(u^{in},\rho^{in},\mu^{in},P^{in})(x,t),& x\in D(t),\\
		(u^{out},\rho^{out},\mu^{out},P^{out})(x,t),& x\in D(t)^c=\R^2\smallsetminus\overline{D(t)}.
	\end{array}\right. \\
\end{equation}

A main interest is the dynamics and the regularity of the common boundary between the fluids $\partial D(t)$. The system is assumed to have initial finite kinetic energy
$$
\int_{\R^2}\rho(x,0)|u(x,0)|^2dx<\infty,
$$      
providing the physically relevant scenario. The classical free boundary physical conditions without capillarity \cite{Denisova2001}
\begin{equation}\label{BC1}
[\![u]\!]=0 \text{ on } \partial D(t),
\end{equation}
\begin{equation}\label{BC2}
[\![\mu\mathbb{D}u-\mathbb{I}_2P]\!]n=0\text{ on } \partial D(t),
\end{equation}
are recovered by considering the equations \eqref{incom}-\eqref{viscosity} in a weak sense, together with the regularity obtained for the solution \cite{GancedoG-J2018}.

\subsection{Previous Results.} Free boundary Navier-Stokes problems have a long history in mathematical science. The one-fluid case (vacuum-fluid interaction where $\mu^{out}=0=\rho^{out}$) was first considered, where global-in-time existence with gravity for near planar initial data was proved. Recently, low regularity results for inhomogeneous Navier-Stokes equations in the whole space have given new approaches for the two-fluid case (fluid-fluid interaction). They consider $\mu>0$, giving global regularity for different scenarios. We describe first the classical vacuum-fluid case and later the fluid-fluid interaction.  

The first study of the free boundary Navier-Stokes equations goes back to \cite{Solonnikov1977} where fluid-vacuum interaction was studied for closed contours with no gravity ($g=0$) using H\"older spaces with the appropriate parabolic scale \cite{Solonnikov1990,Solonnikov1991}. Local well-posedness in Sobolev spaces was given next for the horizontally-flat geometry where the fluid lies essentially on top of a fixed bottom with non-slip boundary condition and below vacuum \cite{Beale1981}.  See \cite{Abels2005} for similar results in $L^p$ Sobolev spaces.

The long time behavior of solutions was studied in \cite{Beale1981}, giving existence up to time $T$ depending on the size of the initial, near horizontally flat, data. The first global-in-time existence result for small initial data was given for the surface tension case \cite{Beale1983}. This result was extended to the case without surface tension \cite{Sylvester1990,TaniTanaka1995}. 
After those results, sharp decay rates of the solution were given in the case of surface tension for asymptotically-flat \cite{BealeNishida1985} and horizontally-periodic geometries \cite{NishidaTeramotoYoshihara2004}. More recently, the results were extended with different approaches without the help of surface tension for both geometries \cite{GuoTice2013-2,GuoTice2013}. See the recent paper \cite{DHMT2020} where global well-posedness is shown for this free boundary value problem with the initial domain the half-space and the initial velocity small with respect to a scaling invariant norm.
Contrarily, large size initial data produces finite-time singularities. Navier-Stokes free boundary blows-up in finite time for the 2D vacuum-fluid interaction case \cite{CCFGG-S2019}. 
The result considers closed contours producing splash singularities (particle collision on the evolving boundary) in finite time. See \cite{CoutandShkoller2019} for the extension of the blow-up to the 3D case. 
%This scenario is given when two different particles on the moving boundary collide at a single physical point meanwhile its regularity is preserved. At time of blow-up the incompressibility condition is not satisfied due to the fact that there exist at least two different labels which represent the same physical point in the space $\R^{2}$. It yielded the first finite time singularity formation for incompressible Navier-Stokes equations. The same result has been extended to the 3D case in \cite{CoutandShkoller2019}.

The techniques in \cite{Solonnikov1977} were extended to the case of two fluids to study the global-in-time well-posedness of the problems for small initial velocity \cite{Denisova2008}. See \cite{SSZ2020} where the low regularity case is considered. In \cite{WTK2014}, decay estimates are obtained for the the internal waves case with gravity.

A different approach to study the interface evolution between immiscible fluids is to use inhomogeneous Navier-Stokes for low regular solutions. Parabolicity can be exploited to gain enough regularity for the velocity in the two-fluid case even when the functions defining the fluid properties are given as in \eqref{jump} . The approaches with no viscosity jump ($\mu=1$) are explained first. In two dimensions, there is global regularity for the system (\ref{incom}-\ref{mass}-\ref{NavierStokes}) for general smooth positive initial density \cite{LadyzenskajaSolonnikov1975}. In the three dimensional case, global regularity for large initial data is open as it contains Navier-Stokes as a particular case \cite{Fefferman2006}. If $0\leq \rho(x,0)\in L^\infty$ is allowed and  $\sqrt{\rho(x,0)}u(x,0)\in L^2(\R^d)$, $d=2,3$, there exist global in time weak solutions satisfying
$$
\int\rho(x,t)|u(x,t)|^2dx+2\int_0^t\int|\nabla u(x,s)|^2dxds\leq \int\rho(x,0)|u(x,0)|^2dx,
$$
with $\rho\in L^\infty((0,T)\times\R^d)$, $\rho u\in L^\infty(0,T;L^2(\R^d))$ and $u\in L^2(0,T;\dot{H}^1(\R^d))$ \cite{Simon1990}. Along the paper, we will use the convention that spaces with a dot denote their homogeneous counterpart. Considering fluids of different constant densities, domains evolving by the fluid velocity were proved to preserve its volume \cite{Lions1996}. On the other hand, the propagation of regularity for the free boundary $\partial  D(t)$ was proposed as a challenging open question in the same book (1996 P.L. Lions' density patch problem). 

Recently, global regularity results in 2D, and with smallness assumptions in 3D, have been obtained for low regular positive density and constant viscosity. Global well-posedness was shown for initial discontinuous densities with sufficiently small jumps and small initial velocities \cite{DanchinMucha2012,DanchinMucha2013}. The case of more regular velocity was considered in \cite{HPaicuZ2013}. Finally, in \cite{PaicuZZ2013} the smallness conditions of the density jump were removed. 
After the results above, global-in-time regularity for fluids of different densities (density patch problem) has been studied. Persistence of $C^{2+\gamma}$ regularity of the free boundary results was shown in 2D for $0<\gamma<1$, using paradifferential calculus and striated regularity techniques. The works consider positive densities with small jump first \cite{LiaoZhang2016} and later without smallness assumption \cite{LiaoZhang2016Pre}. Using the approach in \cite{DanchinMucha2012}, propagation of $C^{1+\gamma}$ regularity was given for small density jump and small initial velocity \cite{DanchinZhang2017}. The size restriction was removed in \cite{GancedoG-J2018}, providing global in time regularity for $C^{1+\gamma}$ 2D contours. This approach does not use pa\-ra\-dif\-fer\-en\-tial calculus but bootstrapping arguments, getting propagation of regularity from weak solutions to $C^{1+\gamma}$. It uses an elliptic approach inspired by previous results obtained for 2D Boussinesq temperature fronts \cite{GancedoG-J2017}. See \cite{LiaoLiu2016} for the 3D extension with high regularity and smallness in velocity and density jump.
In the bounded or periodic case, a new approach has been used to allow the case of possibly vanishing density, with no restriction on the jump size, no gravity and constant viscosity \cite{DanchinMucha2019}. In this density zero scenario, the interface evolution would be driven by a Stokes--Navier-Stokes interaction, dealing with a linear Stokes flow for one of the fluids.

For the more singular case of variable viscosity, with density merely bounded, under the additional assumptions that $u_0\in H^1(\mathbb{T}^2)$ and sufficiently small viscosity variation in $L^\infty$, the weak solutions constructed in \cite{Lions1996} satisfy that $u\in L^\infty(0,T;H^1)$, $\sqrt{\rho} u_t\in L^2(0,T;L^2)$, $\rho,\mu\in L^\infty(0,T;L^\infty)$ for all $T>0$ \cite{Desjardins97}.
 However, uniqueness and regularity of these solutions was not known, unless the initial density and viscosity satisfy certain smoothness (at least slightly more than continuity, see \cite{burtea17} and the references therein). 
Recently,  \cite{PaicuZhang2020} global-in-time regularity for positive  density and small viscosity jump is obtained in $\mathbb{R}^2$ under the additional assumption of certain striated regularity for the initial viscosity. In particular, they showed global-in-time propagation of the $H^{5/2}$ regularity of the moving interface for the density and viscosity patch problem. The strategy of the proof uses paradifferential calculus together with striated regularity estimates. The approach is in the spirit of the global regularity result for the 2D vortex patch problem shown in \cite{Chemin1993}.

\subsection*{Main Result.}
In this paper, we prove global-in-time well-posedness for  the two-dimensional density and viscosity patch problem. We study the evolution of two fluids with different densities and viscosities evolving according to inhomogeneous Navier-Stokes \eqref{incom}-\eqref{viscosity}. The initial density and viscosity functions are bounded from below and from above as follows:
$$
0<\rho^m\leq \rho_0(x)\leq\rho^M,\quad 0<\mu^m\leq \mu_0(x)\leq\mu^M.
$$
The initial interface between the fluids is assumed to be a closed $C^{1+\gamma}$ regular curve in the plane. Specifically, we prove the following result.
\begin{thm}\label{Case1}
 Let $D_0\subset \R^2$ be a bounded  domain whose boundary $\partial D_0$ is non self-intersecting and of class $C^{1+\gamma}$, $0<\gamma<1$. Let $\rho_0^{in}\in C^\gamma(\overline{D}_0)$, $\rho_0^{out}\in C^\gamma(\mathbb{R}^2\setminus D_0)$, with $\rho^{out}_0-\rho^{\infty}$, where $\rho^\infty\in \mathbb{R}_+$, and $\mu=\tilde{\mu}(\rho)$ with $\tilde{\mu}$ smooth.  Let the initial density be given by
\begin{equation*}
\begin{aligned}
\rho_0(x)&=\rho_0^{in}(x) 1_{ D_0}(x)+\rho_0^{out}(x) 1_{D_0^c}(x)>0,
\end{aligned}
\end{equation*}
where $1_{D_0}$ is the characteristic function of $D_0$, and let $u_0\in L^r\cap H^{\gamma+\varepsilon}$, $0<\varepsilon<\min\{\gamma,1-\gamma\}<1$, $1<r<\min\{\frac{2}{2-\gamma+\varepsilon},\frac{2}{1+\gamma}\}$ be a divergence-free vector field.
Then, there exists $\delta>0$ such that if
\begin{equation}\label{smalljump}
    \big\|1-\frac{\mu_0}{\bar{\mu}}\big\|_{L^\infty}\leq \delta,\quad\mbox{with}\quad \quad{\overline{\mu}=\frac{\mu^m+\mu^M}{2}},
\end{equation}
 there exists a unique global solution $(u,\rho,\mu)$  of \eqref{incom}-\eqref{viscosity} with $u(x,0)=u_0(x)$, $\rho(x,0)=\rho_0(x)$ and $\mu(x,0)=\mu_0(x)$ such that
$$u\in C(\mathbb{R}_+;H^{\gamma+\varepsilon})\cap L^1(\mathbb{R}_+;W^{1,\infty})\cap L^1(\mathbb{R}_+;C^{1+\gamma}(\overline{D(t)})\cup C^{1+\gamma}(\mathbb{R}^2\setminus D(t))),$$
$$\partial D\in C(\mathbb{R}_+; C^{1+\gamma}),$$
where $D(t)=X(D_0,t)$, with $X$ the particle trajectories \eqref{particle_trajectories} associated to the velocity field and
$$\rho(x,t)=\rho^{in}(x,t) 1_{D(t)}(x)+\rho^{out}(x,t)1_{D(t)^c}(x), \quad \rho(X(y,t),t)=\rho_0(y).$$
Moreover, for any $t\geq0$, 
\begin{equation*}
    \|\sqrt{\rho}u\|_{L^2}^2(t)+\int_0^t\|\sqrt{\mu}\minspace\mathbb{D}u\|_{L^2}^2d\tau \leq \|\sqrt{\rho_0}u_0\|_{L^2}^2, 
\end{equation*}
\begin{equation*}
    t^{1-\gamma-\varepsilon}\|\nabla u\|_{L^2}^{2}+\int_0^t\tau^{1-\gamma-\varepsilon}\|\sqrt{\rho}D_t u\|_{L^2}^{2}\leq C(\|\sqrt{\rho_0}u_0\|_{L^2},\mumin,a^M,\delta)\|u_0\|_{\dot{H}^{\gamma+\varepsilon}}^2,
\end{equation*}
\begin{equation*}
    t^{2-\gamma-\varepsilon}\|D_t u\|_{L^2}^{2}+\int_0^t \tau^{2-\gamma-\varepsilon}\|\nabla D_t u\|_{L^2}^2\leq C(\|\sqrt{\rho_0}u_0\|_{L^2},a^m,a^M,\delta)\|u_0\|_{\dot{H}^{\gamma+\varepsilon}}^2,
\end{equation*}
and
\begin{equation*}
    \int_0^t\|\nabla u\|_{L^\infty}d\tau+\int_0^t\|\nabla u\|_{\dot{C}^\gamma(\overline{D(t)})\cup \dot{C}^\gamma(\mathbb{R}^2\setminus D(t))}d\tau\leq C,
\end{equation*}
with $C=C(a^{m}, \|a_0\|_{C^\gamma({\overline{D}_0})\cap C^\gamma(\mathbb{R}^2\setminus D_0)},\|a^{out}-a^\infty\|_{L^2},\delta,\|u_0\|_{L^r},\|u_0\|_{H^{\gamma+\varepsilon}})$, and $a\equiv \rho, \mu$.
\end{thm}
Given that $\mu=\tilde{\mu}(\rho)$, along the paper we will use the notation $\muin=\tilde{\mu}(\rho^{in}), \muout=\tilde{\mu}(\rho^{out})$, $\mu^{\infty}=\tilde{\mu}(\rho^{\infty})$, and we will have that $\mu(X(y,t),t)=\mu_0(y)=\tilde{\mu}(\rho_0(y))$, and therefore 
$$\mu(x,t)=\muin(x,t) 1_{D(t)}(x)+\muout (x,t)1_{D(t)^c}(x).$$
The first part of the proof consists in getting a priori estimates which are sharp in Sobolev regularity for the initial velocity to propagate $C^{1+\gamma}$ regularity. This is achieved by introducing time weights and interpolation. Then, a key step in the proof will be to obtain the $L^1$-in-time Lipschitz-in-space estimate for the velocity. This is difficult as the gradient of the velocity is given implicitly by a higher-order Riesz transform applied to a discontinuous function on the moving interface. I.e., this function depends itself on the gradient of the velocity multiplied by the viscosity jump scalar \eqref{aux3I7I8}. We will overcome this difficulty by propagating further regularity on each domain separately. As part of the argument, we will use the following new quantitative estimate.
Consider higher-order Riesz transform operators of even order $2l$, $l\geq1$, given by
	\begin{equation}\label{calderonzygmund}
		\begin{aligned}
			R(f)(x)= \lim_{\varepsilon\to 0}\int_{|x-y|>\varepsilon} K(x-y)f(y)dy,
		\end{aligned}
	\end{equation}
	where
	\begin{equation}\label{kernel}
		K(x)=\frac{P_{2l}(x)}{|x|^{n+2l}},
	\end{equation}
	and $P_{2l}(x)$ is a homogeneous polynomial of degree $2l$ in $\R^2$. Then we have the following result.
\begin{thm}[\cite{GancedoG-J2021}]\label{SingIntHolder}
	Assume $D\subset \mathbb{R}^2$ is a bounded domain of class $C^{1+\gamma}$, $0<\gamma<1$. Then, the Calder\'on-Zygmund operator \eqref{calderonzygmund} with kernel \eqref{kernel} applied to the characteristic function of $ D$, $1_D$, defines a piecewise $C^{\gamma}$ function, \begin{equation*}
	R(1_D)\in C^{\gamma}(\overline{D})\cup C^{\gamma}(\R^2\setminus D).
	\end{equation*}
	Moreover, it satisfies the bound
	\begin{equation*}
	    \begin{aligned}
	        \|R(1_D)\|_{\dot{C}^{\gamma}(\overline{ D})\cup \dot{C}^{\gamma}(\R^2\setminus D)}\leq 
			C \mathcal{P}(\|D\|_{*}\!+\!\|D\|_{\Lip})\|D\|_{\dot{C}^{1+\gamma}}.
	    \end{aligned}
	\end{equation*}
\end{thm}
Above, $\|\cdot\|_*$ measures the arc-chord condition of the boundary of the domain, $\|\cdot\|_{\Lip}$ is the Lipschitz norm, $\|\cdot\|_{\dot{C}^{1+\gamma}}$ is the homogeneous H\"older norm and $\mathcal{P}$ is a polynomial function. If we denote $y(\alpha)$, $\alpha\in [0,2\pi)=\mathbb{T}$, the parametrization of the boundary $\partial D$, these quantities are defined as follows
$$\|D\|_*:=\sup_{\alpha\neq\beta}\frac{|\alpha-\beta|}{|y(\alpha)-y(\beta)|},\quad\|D\|_{\Lip}:=\sup_{\alpha\neq\beta}\frac{|y(\alpha)-y(\beta)|}{|\alpha-\beta|},
\quad \|D\|_{\dot{C}^{1+\sigma}}:=\sup_{\alpha\neq\beta}\frac{|y'(\alpha)- y'(\beta)|}{|\alpha-\beta|^{\sigma}}.
$$
\begin{remark}
	By the boundary condition \eqref{BC2}, one cannot expect to obtain globally-in-space further regularity  than $\nabla u\in L^\infty(\mathbb{R}^2)$. Indeed, if we denote by $\tau$ and $n$ the tangent and normal vectors to the boundary, we have that
	\begin{equation*}
	[\![\mu \mathbb{D}_{ij}u-P\delta_{ij}]\!]n_j=0\Rightarrow
	\left\{\begin{aligned}
	[\![\mu\minspace  n\cdot\mathbb{D}u \cdot n ]\!]&=[\![P]\!],\\
	[\![\mu\minspace \tau\cdot\mathbb{D}u \cdot n ]\!]&=0,
	\end{aligned}\right. 
	\end{equation*}
	thus, if $\nabla u$ were continuous, then we would obtain $[\![\mu]\!]=0$.
\end{remark}

	\vspace{0.5cm}

\noindent\textbf{Outline of the paper:} The rest of the paper is structured as follows. The proof of the main Theorem \ref{Case1} is divided into existence and uniqueness. For the existence, we proceed to obtain the necessary \textit{a priori} estimates. We separate the process into six steps, bootstrapping the regularity obtained from one to the next. Steps 1-3 consist in energy estimates with time weights, which allow to obtain high regularity for the velocity despite the low regularity of the density, viscosity, and initial velocity. Step 4 bounds the crucial $L^1$-in-time Lipschitz regularity of the velocity in terms of the higher H\"older regularity on each side, which is studied in Step 5. The previous steps are combined with quantitative estimates of even singular integral operators acting on $C^{1+\gamma}$ domains in Step 6. This concludes the proof of existence. Next, the uniqueness of solutions is shown. The proof is done in Lagrangian variables, due to the discontinuity jumps of the density and viscosity across the fluid interface.

\section{Proof of Theorem \ref{Case1}}\label{sec:2}

\vspace{0.2cm}

\subsection{Existence:} The proof of existence follows a standard mollifier and compactness argument (see e.g. \cite{DanchinMucha2019}, \cite{PaicuZhang2020}). Once the initial data is smoothed out, we show the \textit{a priori} estimates for the corresponding unique smooth solution. The $L^1$-in-time Lipschitz-in-space estimate for the velocity implies that the solution exists globally in time. The fact that all these estimates will be uniform in the mollifying parameter gives the necessary compactness to pass to the limit. We proceed to obtain the \textit{a priori} estimates. 

\vspace{0.5cm}

\noindent \textbf{Step 1:} $\sqrt{\rho}u\in L^\infty(0,T;L^2), \minspace\sqrt{\mu}\minspace\mathbb{D}u\in L^2(0,T;L^2)$

\vspace{0.3cm}

 We first obtain the $L^2$ energy balance
\begin{equation*}
\begin{aligned}
\frac{1}{2}\frac{d}{dt}\intrtwo \rho |u|^2dx &\!=\!-\!\intrtwo \!\!\mu \pj   u_i \left(\pj   u_i\!+\!\pai u_j\right)dx=\!-2\!\intrtwo \!\mu \Big((\pone u_1)^2\!+\!(\ptwo u_2)^2\!+\!\frac12(\pone u_2\!+\!\ptwo u_1)^2\Big)dx\\
&=-\frac12 \|\sqrt{\mu} \minspace \mathbb{D}u\|_{L^2},
\end{aligned}
\end{equation*}
which after integration in time reads as follows
\begin{equation}
\label{l2balancesym}
\|\sqrt{\rho}u\|_{L^2}^2(t)+\int_0^t\|\sqrt{\mu} \minspace \mathbb{D}u\|_{L^2}^2(\tau)d\tau \leq \|\sqrt{\rho_0}u_0\|_{L^2}^2.
\end{equation}

\noindent \textbf{Step 2:} $t^{\frac{1-\gamma-\varepsilon}2}\nabla u\in L^\infty(0,T;L^2), \minspace t^{\frac{1-\gamma-\varepsilon}2}D_t u\in L^2(0,T;L^2)$
\vspace{0.3cm}

To obtain the result with low regularity initial data, we use an interpolation argument and time-weighted energy estimates \cite{PaicuZZ2013}. Consider the linearized problem
\begin{equation*}
    \begin{aligned}
    \rho(v_t+u\cdot\nabla v)&=\nabla\cdot(\mu\mathbb{D}v-\mathbb{I}_2P),\\
    \rho_t&=-u\cdot\nabla \rho.
    \end{aligned}
\end{equation*}
It holds that
\begin{equation}
\label{l2balancesymv}
\|\sqrt{\rho}v\|_{L^2}^2(t)+\int_0^t\|\sqrt{\mu} \minspace \mathbb{D}v\|_{L^2}^2(\tau)d\tau \leq \|\sqrt{\rho_0}v_0\|_{L^2}^2.
\end{equation}
Next, we take inner product of \eqref{NavierStokes} with $D_tv:=v_t+u\cdot\nabla v$ and then integrate by parts to obtain that
\begin{equation*}
\begin{aligned}
\intrtwo \rho |D_tv|^2dx&=-\intrtwo \pj  D_tv_i \left(\mu\mathbb{D}_{ij}v-P\delta_{ij}\right)dx.
\end{aligned}
\end{equation*}
By the commutator 
\begin{equation*}
\left[D_t,\pj  \right]f=-\pj   u\cdot \nabla f,
\end{equation*}
and the incompressibility condition, it follows that
\begin{equation*}
\begin{aligned}
\intrtwo \rho |D_tv|^2dx&=-\intrtwo D_t \pj  v_i \left(\mu\mathbb{D}_{ij}v-P\delta_{ij}\right)dx-\intrtwo  \pj  u_k\pk  v_i \left(\mu\mathbb{D}_{ij}v-P\delta_{ij}\right)dx\\
&=-\intrtwo D_t \pj  v_i \mu\mathbb{D}_{ij}v\minspace dx-\intrtwo  \pj u_k\pk  v_i \left(\mu\mathbb{D}_{ij}v-P\delta_{ij}\right)dx.
\end{aligned}
\end{equation*}
Noticing that $D_t\mu=0$, we introduce a time weight $t$ followed by integration in time,
\begin{equation}\label{balance2}
\begin{aligned}
\frac{t}2\|\sqrt{\mu}\minspace \mathbb{D} v\|_{L^2}^2(t)+\int_0^t \tau\|\sqrt{\rho} D_tv\|_{L^2}^2(\tau)d\tau&=\frac12\int_0^t\|\sqrt{\mu}\minspace \mathbb{D} v\|_{L^2}^2d\tau- \int_0^t\tau\intrtwo  \mu\pj  u_k\pk  v_i \mathbb{D}_{ij}v dx d\tau\\
&\quad+\int_0^t\tau\intrtwo  \pai u_k\pk  v_i P\minspace  dxd\tau.
\end{aligned}
\end{equation}
We take the divergence of \eqref{NavierStokes} to obtain the following expression for the pressure:
\begin{equation}\label{pressure}
P=\left(-\Delta\right)^{-1}\nabla\cdot \left(\rho D_tv\right)-\nabla\cdot\nabla\cdot\left(-\Delta\right)^{-1}\left(\mu\mathbb{D}v\right).
\end{equation}
Substituting \eqref{pressure} in \eqref{balance2} we have that
\begin{equation}\label{balance22}
\frac{t}2\|\sqrt{\mu}\minspace \mathbb{D} v\|_{L^2}^2(t)+\int_0^t \tau \|\sqrt{\rho} D_tv|_{L^2}^2(\tau)d\tau= \frac12\int_0^t\|\sqrt{\mu}\minspace \mathbb{D} v\|_{L^2}^2d\tau+ I_1+I_2+I_3,
\end{equation}
where
\begin{equation*}
\begin{aligned}
I_1&=-\int_0^t\tau\intrtwo  \mu\pj  u_k\pk  v_i \mathbb{D}_{ij}v \hspace{0.05cm}dxd\tau,\\
I_2&=\int_0^t\tau\intrtwo \pai u_k\pk v_i \left(-\Delta\right)^{-1}\nabla\cdot \left(\rho D_tv\right)dxd\tau,\\
I_3&=-\int_0^t\tau\intrtwo \pai u_k\pk v_i\nabla\cdot\nabla\cdot\left(-\Delta\right)^{-1}\left(\mu\mathbb{D}v\right)dxd\tau.
\end{aligned}
\end{equation*}
We need estimates for the gradient of $v$ in terms of $D_tv$. Notice that the following identity holds
\begin{equation*}
\begin{aligned}
\nabla\cdot (\mu\mathbb{D}v)=\mum \Delta v+\nabla\cdot (\mu\mathbb{D}v-\mum\minspace\mathbb{D}v),
\end{aligned}
\end{equation*}
where $\bar{\mu}$ can be taken as $\bar{\mu}=(\mumax+\mumin)/2$.
Then,
\begin{equation*}
\begin{aligned}
\nabla v=\frac1{\mum} \nabla \Delta^{-1}\mathbb{P}\nabla\cdot (\mu \mathbb{D}v)-\nabla\Delta^{-1}\mathbb{P}\nabla\cdot \left(\Big(\frac{\mu}{\mum}-1\Big)\mathbb{D}v\right),
\end{aligned}
\end{equation*}
where $\mathbb{P}$ denotes the Leray projector,
\begin{equation*}
    \mathbb{P}f=f-\nabla\Delta^{-1}\nabla\cdot f.
\end{equation*}
Therefore, given condition \eqref{smalljump}, the boundedness of singular integrals in $L^q$, $1<q<\infty$, gives that
\begin{equation}\label{Lp}
\|\nabla v\|_{L^p}\leq c(\delta) \|\nabla\Delta^{-1}\mathbb{P}\nabla\cdot (\mu \mathbb{D}v)\|_{L^p},\qquad 2\leq p\leq \max\big\{\frac{2}{1-\gamma-\varepsilon},\frac{2}{\gamma-\varepsilon}\big\}.
\end{equation}	
Applying the Leray projector to \eqref{NavierStokes} we find the relationship between $\mathbb{D}u$ and $D_tu$
\begin{equation}\label{elliptic}
\mathbb{P}\left(\rho D_tv\right)=\mathbb{P}\nabla\cdot\left(\mu \mathbb{D}v\right).
\end{equation}
Recalling the following Gagliardo-Nirenberg inequality in $\mathbb{R}^2$, 
\begin{equation}\label{L4}
    \|f\|_{L^p}\leq c\|f\|_{L^2}^{\frac{2}{p}}\|\nabla f\|_{L^2}^{1-\frac{2}{p}},
    \end{equation}
followed by \eqref{elliptic}, one can find from \eqref{Lp} that
\begin{equation}\label{Lpv}
\|\nabla v\|_{L^p}\leq c(\delta)\|\mu\minspace\mathbb{D}v\|_{L^2}^{\frac2p} \|\nabla \nabla\Delta^{-1} \mathbb{P}(\rho D_t v)\|_{L^2}^{1-\frac2p}\leq c(\delta)\|\mu\minspace\mathbb{D}v\|_{L^2}^{\frac2p}\|\rho D_t v\|_{L^2}^{1-\frac2p}.
\end{equation}
In particular, 
\begin{equation}\label{L3}
\begin{aligned}
\|\nabla v\|_{L^4}\leq c(\delta)\|\mu\minspace\mathbb{D}v\|_{L^2}^{1/2}\|\rho D_t v\|_{L^2}^{1/2}.
\end{aligned}
\end{equation}
Thus, the terms $I_1$ and $I_3$ are readily bounded as follows
\begin{equation*}
\begin{aligned}
I_1+I_3&\leq c\int_0^t \tau\|\nabla u\|_{L^2}\|\nabla v\|_{L^4}\|\mathbb{D}v\|_{L^4}d\tau\\
&\leq \frac14\int_0^t\tau\|\sqrt{\rho}D_t v\|_{L^2}^2d\tau+c(\delta)\int_0^t\tau\|\nabla u\|_{L^2}^2 \|\sqrt{\mu}\minspace\mathbb{D}v\|_{L^2}^2d\tau.
\end{aligned}
\end{equation*}
Then, denoting $\mathcal{H}^1$ the Hardy space, we get for $I_2$ \eqref{balance22} the estimate
\begin{equation*}
    \begin{aligned}
    I_2&\leq c\int_0^t\tau \|\pai u_k\pk v_i\|_{\mathcal{H}^1}\|\left(-\Delta\right)^{-1}\nabla\cdot \left(\rho D_tv\right)\|_{BMO} \minspace d\tau.
    \end{aligned}
    \end{equation*}
Since for each $i$ the term $\partial_i u\cdot \nabla v_i$ is the product of a divergence-free function and a curl-free one, we can apply the div-curl lemma to get
\begin{equation*}
    \|\pai u_k\pk v_i\|_{\mathcal{H}^1}\leq C\|\nabla u\|_{L^2}\|\nabla v\|_{L^2},
\end{equation*} which together with  the embedding $\dot{H}^1\hookrightarrow BMO$ gives that
\begin{equation*}
    \begin{aligned}
    I_2   &\leq c\int_0^t\tau\|\nabla u\|_{L^2}\|\nabla v\|_{L^2}\|\rho D_t v\|_{L^2}d\tau\\
    &\leq\frac14\int_0^t\tau\|\sqrt{\rho} D_t v\|_{L^2}^2d\tau+c\int_0^t\tau\|\nabla u\|_{L^2}^2\|\mathbb{D}v\|_{L^2}^2d\tau.
    \end{aligned}
\end{equation*} 
Therefore, we have that \eqref{balance22} becomes
\begin{equation*}
\begin{aligned}
t\|\sqrt{\mu}\minspace\mathbb{D} v\|_{L^2}^2(t)+\int_0^t \tau\|\sqrt{\rho} D_t v\|_{L^2}^2(\tau)d\tau&\leq  \|\sqrt{\rho_0}v_0\|_{L^2}^2+c(\mumin,\delta)\int_0^t \tau\|\sqrt{\mu}\minspace\mathbb{D} v\|_{L^2}^2\|\sqrt{\mu}\minspace\mathbb{D} u\|_{L^2}^2d\tau.
\end{aligned}
\end{equation*}
Gr\"onwall's lemma followed by \eqref{l2balancesym}
yields the balance
\begin{equation}\label{balanceDvL2}
\begin{aligned}
t\|\sqrt{\mu}\minspace\mathbb{D}v\|_{L^2}^2+\int_0^t \tau\|\sqrt{\rho} D_t v\|_{L^2}^2(t) d\tau\leq c(\mumin,\delta,\|\sqrt{\rho_0}u_0\|_{L^2}) \|\sqrt{\rho_0}v_0\|_{L^2}^2.
\end{aligned}
\end{equation}
We can repeat the steps above without the time weight to obtain that
\begin{equation}\label{balanceDvH1}
\begin{aligned}
\|\sqrt{\mu}\minspace\mathbb{D}v\|_{L^2}^2+\int_0^t \|\sqrt{\rho} D_t v\|_{L^2}^2(\tau) d\tau\leq c(\mumin,\delta,\|\sqrt{\rho_0}u_0\|_{L^2}) \|\sqrt{\mu_0}\minspace\mathbb{D}v_0\|_{L^2}^2.
\end{aligned}
\end{equation}
Thus, the linear operator $Tv_0=\nabla v$ satisfies the bounds $\|Tv_0\|_{L^2}\leq c\|\nabla v_0\|_{L^2}$ and $\|Tv_0\|_{L^2}\leq ct^{-\frac12}\| v_0\|_{L^2}$, and hence we conclude that
\begin{equation*}
    \begin{aligned}
        \|\nabla u\|_{L^2}\leq c(\mumin,\delta,\|\rho_0u_0\|_{L^2})t^{\frac{-1+\gamma+\varepsilon}{2}}\|v_0\|_{\dot{H}^{\gamma+\varepsilon}}.
    \end{aligned}
\end{equation*}
Using similarly Stein's interpolation theorem as in \cite{PaicuZZ2013} for the terms with time integrals, we close the balance in $\dot{H}^{\gamma+\varepsilon}$,
\begin{equation}\label{balanceHg}
\begin{aligned}
t^{1-\gamma-\varepsilon}\|\nabla u\|_{L^2}^2+\int_0^t \tau^{1-\gamma-\varepsilon}\|\sqrt{\rho} D_t u\|_{L^2}^2(\tau) d\tau\leq c(\mumin,\delta,\|\sqrt{\rho_0}u_0\|_{L^2}) \|u_0\|_{\dot{H}^{\gamma+\varepsilon}}^2.
\end{aligned}
\end{equation}
Notice that we can combine \eqref{balanceDvL2} and \eqref{balanceHg} to obtain
\begin{equation}\label{balanceNablau1}
\max\{t^{1-\gamma-\varepsilon},t\}\|\nabla u\|_{L^2}^2+\int_0^t \max\{\tau^{1-\gamma-\varepsilon},\tau\}\|\sqrt{\rho} D_t u\|_{L^2}^2(\tau) d\tau\leq C.
\end{equation}
We will need further time decay. We use the following theorem:
\begin{thm}[\cite{HuangPaicu2014}] \label{HuangPaicuThm}
For $1<r<2$, and $0<\alpha<1$, let $u_0\in L^r\cap H^\alpha$, $a_0-a^\infty\in L^2$, and $0<a^m<a_0\in L^\infty$ with $a\equiv\rho,\mu$. Then, under the assumption of small viscosity contrast, inhomogeneous Navier-Stokes  with initial data $(\rho_0, \mu_0, u_0)$ has a global weak solution and there exists a constant $C_\alpha$ which depends on $\|\rho_0-\rho^\infty\|_{L^2}$, $\|u_0\|_{L^r}$, and $\|u_0\|_{H^\alpha}$, such that there hold
\begin{equation*}
    \begin{aligned}
    \|u(t)\|_{L^2}^2&\leq C_\alpha (t+e)^{-\frac{2}{r}+1},\qquad \|\nabla u(t)\|_{L^2}^2\leq C_\alpha (t+e)^{-\frac{2}{r}+\varepsilon},\\
    \int_0^\infty t^{1-\kappa}(t+e)^{\kappa+\frac{2}{r}-1-\varepsilon}&\|u_t\|_{L^2}^2+\|\mathbb{P}\minspace\text{\em{div}}(\mu \mathbb{D}u)\|_{L^2}^2+\|(\mathbb{I}_2-\mathbb{P})\minspace\text{\em{div}}(\mu \mathbb{D}u)-\nabla P\|_{L^2}^2dt\leq C_\alpha,
    \end{aligned}
\end{equation*}
with any $0<\varepsilon<1$ and $0<\kappa<\alpha$.
\end{thm}
Notice that by \eqref{Lpv} and Young's inequality we have that
\begin{equation*}
    \|D_t u\|_{L^2}^2\leq c(\|u_t\|_{L^2}^2+\|u\|_{L^2}^2\|\nabla u\|_{L^2}^4).
\end{equation*}
Hence by \eqref{balanceHg}, the estimates in Theorem 2.1 above and \eqref{l2balancesym} it is possible to get
\begin{equation*}
    \int_0^t \max\{\tau^{1-\gamma-\varepsilon},\tau^{\frac{2}{r}-\varepsilon}\}\|D_t u\|_{L^2}^2(\tau)d\tau\leq C,
\end{equation*}
with $C=C(a^m, a^M, \|a_0-a^\infty\|_{L^2},\delta,\|u_0\|_{L^r}, \|u_0\|_{H^{\gamma+\varepsilon}})$.
Thus, since $2/r-\varepsilon>1$, we can improve \eqref{balanceNablau1} for large times,
\begin{equation}\label{balanceNablau}
\begin{aligned}
\max\{t^{1-\gamma-\varepsilon},t^{\frac{2}{r}-\varepsilon}\}\|\nabla u\|_{L^2}^2+\int_0^t \max\{\tau^{1-\gamma-\varepsilon},\tau^{\frac{2}{r}-\varepsilon}\}\|\sqrt{\rho} D_t u\|_{L^2}^2(\tau) d\tau\leq C.
\end{aligned}
\end{equation}

\vspace{0.3cm}

\noindent \textbf{Step 3:} $t^{1-\frac{\gamma+\varepsilon}2}D_t u\in L^\infty(0,T;L^2), \minspace t^{1-\frac{\gamma+\varepsilon}2}\nabla D_t u\in L^2(0,T;L^2), \minspace u\in C(\mathbb{R}_+;H^{\gamma+\varepsilon})$

\vspace{0.3cm}

We proceed to obtain higher regularity estimates for $D_t u$. We take $D_t$ in \eqref{NavierStokes} and then inner product with $D_t u$ to obtain that
\begin{equation*}
\intrtwo D_t u \cdot \rho D_t^2u\hspace{0.05cm}dx=\intrtwo D_t u \cdot D_t \nabla\cdot (\mu\mathbb{D}u)\hspace{0.05cm}dx-\intrtwo D_t u\cdot D_t \nabla P\hspace{0.05cm}dx,
\end{equation*} 
which after multiplication by the time weight $t^{2-\gamma-\varepsilon}$ gives that 
\begin{equation}\label{balanceDtu}
\begin{aligned}
\frac12\frac{d}{dt}\Big(t^{2-\gamma-\varepsilon}\intrtwo \rho |D_t u|^2\hspace{0.05cm}dx\Big) &=\frac{2\!-\!\gamma\!-\!\varepsilon}2t^{1-\gamma-\varepsilon}\|\sqrt{\rho}D_tu\|_{L^2}^2+I_4+I_5+I_6,
\end{aligned}
\end{equation}
with
$$
I_4=t^{2-\gamma-\varepsilon}\intrtwo D_t u_i \pj   D_t  (\mu\mathbb{D}_{ij}u)\hspace{0.05cm}dx,\qquad
I_5=
-t^{2-\gamma-\varepsilon}\intrtwo D_t u_i \pj   u_k\pk    (\mu\mathbb{D}_{ij}u)\hspace{0.05cm}dx,$$
and
$$
I_6=-t^{2-\gamma-\varepsilon}\intrtwo D_t u\cdot D_t \nabla P\hspace{0.05cm}dx.
$$
Integration by parts in $I_4$ provides that
\begin{equation*}
\begin{aligned}
I_4&=-t^{2-\gamma-\varepsilon}\intrtwo \pj   D_t u_i \mu D_t\mathbb{D}_{i j}u\hspace{0.05cm}dx\\
&=-t^{2-\gamma-\varepsilon}\intrtwo \pj   D_t u_i \mu (\pj   D_t u_i+\pai D_t u_j)\hspace{0.05cm}dx+t^{2-\gamma-\varepsilon}\intrtwo \pj   D_t u_i \mu (\pj  u_k\pk u_i\!+\!\pai u_k\pk u_j)\hspace{0.05cm}dx\\
&\leq -\frac{t^{2-\gamma-\varepsilon}}{2}\|\sqrt{\mu}\minspace\mathbb{D}D_t u\|_{L^2}^2+2t^{2-\gamma-\varepsilon}\|\sqrt{\mu}\minspace\mathbb{D} D_t u\|_{L^2}\|\nabla u\|_{L^4}^2,
\end{aligned}
\end{equation*}
which by \eqref{L3} (taking $v$ equal to $u$) gives
\begin{equation*}
    I_4\leq -\frac{t^{2-\gamma-\varepsilon}}4\|\sqrt{\mu}\minspace\mathbb{D}D_t u\|_{L^2}^2+c(\delta)\minspace t^{2-\gamma-\varepsilon}\|\mu\mathbb{D}u\|_{L^2}^2\|\rho D_tu\|_{L^2}^2.
\end{equation*}
Integration by parts in term $I_5$ yields
\begin{equation*}
I_5\leq 2\mumax t^{2-\gamma-\varepsilon}\|\nabla D_t u\|_{L^2}\|\nabla u\|_{L^4}^2.
\end{equation*}
The identity
\begin{equation*}
    \partial_k f_i=\pk\Delta^{-1}\pj\mathbb{D}_{ij}f-\nabla\cdot\Delta^{-1}\pk\pai f
\end{equation*}
and the fact that
\begin{equation*}
    \nabla\cdot D_t u=\nabla u\cdot \nabla u,
\end{equation*}
implies that
\begin{equation*}
    \|\nabla D_t u\|_{L^2}^2=\|\mathbb{D}D_t u\|_{L^2}^2+\|\nabla u\cdot\nabla u\|_{L^2}^2.
\end{equation*}
Therefore, applying Young's inequality we obtain
\begin{equation*}
    I_5\leq \frac{t^{2-\gamma-\varepsilon}}8\|\sqrt{\mu}\minspace\mathbb{D}D_t u\|_{L^2}^2+c(\mumin)t^{2-\gamma-\varepsilon}\|\nabla u\|_{L^4}^4.
\end{equation*}
so, using again \eqref{L3}, we have that
\begin{equation}\label{I4I5bound}
    I_4+I_5\leq -\frac{t^{2-\gamma-\varepsilon}}8\|\sqrt{\mu}\minspace\mathbb{D}D_t u\|_{L^2}^2+c(\mumin,\delta)t^{2-\gamma-\varepsilon}\|\mu\mathbb{D}u \|_{L^2}^2\|\rho D_t u\|_{L^2}^2.
\end{equation}
For the $I_6$ term, we first split it as follows
\begin{equation}\label{I6splt}
\begin{aligned}
I_6&=-t^{2-\gamma-\varepsilon}\intrtwo D_tu \cdot  \nabla D_t P\hspace{0.05cm}dx-t^{2-\gamma-\varepsilon}\intrtwo D_tu_i\pai  u_k \pk  P\hspace{0.05cm}dx\\
&= J_1+J_2.
\end{aligned}
\end{equation}
We proceed with $J_2$ first. We substitute the expression \eqref{pressure} for the pressure and integrate by parts to obtain that
\begin{equation*}
\begin{aligned}
J_2&=t^{2-\gamma-\varepsilon}\intrtwo \pk  D_tu_i \pai  u_k (-\Delta)^{-1}\nabla\cdot\nabla\cdot (\mu\mathbb{D}u)\hspace{0.05cm}dx\\
&\quad-t^{2-\gamma-\varepsilon}\intrtwo D_t u_i\pai u_k\pk  (-\Delta)^{-1}\nabla\cdot (\rho D_tu)\hspace{0.05cm}dx\\
&=K_1+K_2.
\end{aligned}
\end{equation*}
The first term $K_1$ is bounded as the previous term $I_5$,
\begin{equation*}
    K_1\leq\frac{t^{2-\gamma-\varepsilon}}{64}\|\sqrt{\mu}\minspace\mathbb{D}D_t u\|_{L^2}^2+c(\mumin,\delta)t^{2-\gamma-\varepsilon}\|\mu\mathbb{D}u \|_{L^2}^2\|\rho D_t u\|_{L^2}^2.
\end{equation*}
For the second one, we integrate by parts twice to get
\begin{equation*}
\begin{aligned}
K_2=&t^{2-\gamma-\varepsilon}\intrtwo \pai u_l\partial_l u_i  u_k\pk  (-\Delta)^{-1}\nabla\cdot (\rho D_tu)\minspace dx\\
&-t^{2-\gamma-\varepsilon}\intrtwo\pk D_tu_i u_k\pai (-\Delta)^{-1}\nabla\cdot(\rho D_tu)\minspace dx.
\end{aligned}
\end{equation*}
Then, 
\begin{equation*}
\begin{aligned}
K_2&\leq c\minspace t^{2-\gamma-\varepsilon}\Big(\|\nabla u\|_{L^4}^{2}\|u\|_{L^4}\|\rho D_tu\|_{L^4}+\|\nabla D_t u\|_{L^2}\|u\|_{L^4}\|\rho D_t u\|_{L^4}\Big),
\end{aligned}
\end{equation*}
so, using repeatedly \eqref{L4} and \eqref{L3}, we can bound it by
\begin{equation*}
\begin{aligned}
K_2&\leq c(\delta)\minspace t^{2-\gamma-\varepsilon}\Big(\|D_tu\|_{L^2}^{\frac32}\|\nabla D_t u\|_{L^2}^{\frac12}\|\nabla u\|_{L^2}^{\frac32}\|u\|_{L^2}^{\frac12}+\|\nabla D_t u\|_{L^2}^{\frac32}\|D_tu\|_{L^2}^{\frac12}\|u\|_{L^2}^{\frac12}\|\nabla u\|_{L^2}^{\frac12}\Big)\\
&\leq \frac{t^{2-\gamma-\varepsilon}}{64}\|\sqrt{\mu}\minspace\mathbb{D} D_tu\|_{L^2}^{2}+c(\mumin,\delta)\minspace t^{2-\gamma-\varepsilon}\|D_tu\|_{L^2}^{2}\|\nabla  u\|_{L^2}^{2}\big(\|u\|_{L^2}^{\frac23}+\|u\|_{L^2}^{2}\big),
\end{aligned}
\end{equation*}
thus 
\begin{equation}\label{J2}
    J_2\leq \frac{t^{2-\gamma-\varepsilon}}{32}\|\sqrt{\mu}\minspace\mathbb{D} D_tu\|_{L^2}^{2}+c(\mumin,\delta)\minspace t^{2-\gamma-\varepsilon}\|D_tu\|_{L^2}^{2}\|\nabla u\|_{L^2}^{2}.
\end{equation}
We proceed with $J_1$ \eqref{I6splt}.
After integration by parts, the term $J_1$ can be written as follows
\begin{equation}\label{J1split}
\begin{aligned}
J_1&=t^{2-\gamma-\varepsilon}\intrtwo \pai u_k\pk  u_i D_t P\minspace dx\\
&=\frac{d}{dt}\Big(t^{2-\gamma-\varepsilon}\intrtwo \pai u_k \pk  u_i P\minspace dx\Big) -t^{2-\gamma-\varepsilon}\intrtwo  D_t (\pai  u_k \pk  u_i)P\minspace dx\\
&\quad-(2-\gamma-\varepsilon)t^{1-\gamma-\varepsilon}\intrtwo \pai u_k\pk u_i P\minspace dx\\
&=K_3+K_4+K_5.
\end{aligned}
\end{equation}
Commuting the time derivative, the term $K_4$ is given by
\begin{equation}\label{k4split}
\begin{aligned}
K_4&=-2t^{2-\gamma-\varepsilon}\intrtwo  \pai  D_t u_k \pk  u_i P\minspace dx+2t^{2-\gamma-\varepsilon}\intrtwo \pai  u_j\pj   u_k\pk  u_i P\minspace dx\\
&=L_1+L_2.
\end{aligned}
\end{equation}
The term $L_1$ is bounded as $J_2$ \eqref{I6splt},
\begin{equation*}
    L_1\leq \frac{t^{2-\gamma-\varepsilon}}{64}\|\sqrt{\mu}\minspace\mathbb{D} D_tu\|_{L^2}^{2}+c(\mumin,\delta)\minspace t^{2-\gamma-\varepsilon}\|D_tu\|_{L^2}^{2}\|\nabla u\|_{L^2}^{2}.
\end{equation*}
Next, we substitute the pressure \eqref{pressure} in $L_2$ to obtain that
\begin{equation*}
\begin{aligned}
L_2&\leq c\minspace t^{2-\gamma-\varepsilon}\|\nabla u\|_{L^4}^{4}+2t^{2-\gamma-\varepsilon}\intrtwo \pai u_j\pj u_k \pk u_i (-\Delta)^{-1}\nabla\cdot(\rho D_tu)\minspace dx.
\end{aligned}
\end{equation*}
Then we note that by integrating by parts twice, the second term can be written as follows
\begin{equation*}
\begin{aligned}
t^{2-\gamma-\varepsilon}\!\intrtwo \!\!\pai u_j\pj u_k \pk u_i (-\Delta)^{-1}\nabla\cdot(\rho D_tu)\minspace &dx=\frac{t^{2-\gamma-\varepsilon}}2\!\intrtwo \!\!\pk u_j\pj u_k  u_i\pai (-\Delta)^{-1}\nabla\cdot(\rho D_tu)\minspace dx,
\end{aligned}
\end{equation*}
and therefore it is bounded as the first term in $K_2$ above. We conclude that
\begin{equation*}
\begin{aligned}
L_2&\leq \frac{t^{2-\gamma-\varepsilon}}{64}\|\sqrt{\mu}\minspace\mathbb{D} D_tu\|_{L^2}^{2}+c(\mumin,\delta)\minspace t^{2-\gamma-\varepsilon}\|D_tu\|_{L^2}^{2}\|\nabla  u\|_{L^2}^{2}(1+\|u\|_{L^2}^{\frac23}),
\end{aligned}
\end{equation*}
and thus
\begin{equation}\label{k4bound}
\begin{aligned}
K_4&\leq \frac{t^{2-\gamma-\varepsilon}}{32}\|\sqrt{\mu}\minspace\mathbb{D} D_tu\|_{L^2}^{2}+c(\mumin,\delta)\minspace t^{2-\gamma-\varepsilon}\|D_tu\|_{L^2}^{2}\|\nabla  u\|_{L^2}^{2}.
\end{aligned}
\end{equation}
Substitution of the expression \eqref{pressure} for the pressure in $K_5$ \eqref{J1split} and integration by parts gives that
\begin{equation*}
    \begin{aligned}
    K_5&=(2\!-\!\gamma\!-\!\varepsilon)t^{1-\gamma-\varepsilon}\intrtwo \pai u_k\pk u_i \nabla\cdot\nabla\cdot\left(-\Delta\right)^{-1}\left(\mu\mathbb{D}v\right)\minspace dx\\
    &\quad+(2\!-\!\gamma\!-\!\varepsilon)t^{1-\gamma-\varepsilon}\intrtwo \pai u_k  u_i \pk\left(-\Delta\right)^{-1}\nabla\cdot \left(\rho D_tv\right)\minspace dx,
    \end{aligned}
\end{equation*}
so using \eqref{L3}, \eqref{L4},
\begin{equation}\label{K5bound}
    \begin{aligned}
    K_5&\leq c(\delta)\minspace t^{1-\gamma-\varepsilon}\Big(\|\nabla u\|_{L^2}^{2}\|\rho D_t u\|_{L^2}+\|\rho D_t u\|_{L^2}^{\frac32}\|\nabla u\|_{L^2}\|u\|_{L^2}^{\frac12}\Big)\\
    &\leq c(\delta)\minspace t^{1-\gamma-\varepsilon}\Big(\|\sqrt{\rho} D_t u\|_{L^2}^{2}+\|\nabla u\|_{L^2}^{4}(1+\|u\|_{L^2}^{2})\Big)\\
     &\leq c(\delta)\minspace t^{1-\gamma-\varepsilon}\Big(\|\sqrt{\rho} D_t u\|_{L^2}^{2}+\|\nabla u\|_{L^2}^{4}\Big).
    \end{aligned}
\end{equation}
Going back to \eqref{J1split}, bounds \eqref{k4bound} and \eqref{K5bound} provide
\begin{equation*}
    \begin{aligned}
    J_1&\leq \frac{d}{dt}\Big(t^{2-\gamma-\varepsilon}\intrtwo \pai u_k \pk  u_i P\minspace dx\Big)+\frac{t^{2-\gamma-\varepsilon}}{32}\|\sqrt{\mu}\minspace\mathbb{D} D_tu\|_{L^2}^{2}+c(\mumin,\delta)\minspace t^{2-\gamma-\varepsilon}\|D_tu\|_{L^2}^{2}\|\nabla  u\|_{L^2}^{2}\\
    &\quad+c(\delta)\minspace t^{1-\gamma-\varepsilon}\Big(\|D_t u\|_{L^2}^{2}+\|\nabla u\|_{L^2}^{4}\Big).
    \end{aligned}
\end{equation*}
Recalling the bound for $J_2$ \eqref{J2}, we obtain for $I_6$ \eqref{I6splt} the following
\begin{equation}\label{I6bound}
    \begin{aligned}
    I_6&\leq \frac{d}{dt}\Big(t^{2-\gamma-\varepsilon}\intrtwo \pai u_k \pk  u_i P\minspace dx\Big)+\frac{t^{2-\gamma-\varepsilon}}{16}\|\sqrt{\mu}\minspace\mathbb{D} D_tu\|_{L^2}^{2}\\
    &\quad+c(\mumin,\delta)\minspace t^{2-\gamma-\varepsilon}\|D_tu\|_{L^2}^{2}\|\nabla  u\|_{L^2}^{2}+c(\delta)\minspace t^{1-\gamma-\varepsilon}\Big(\|D_t u\|_{L^2}^{2}+\|\nabla u\|_{L^2}^{4}\Big).
    \end{aligned}
\end{equation}
Finally, we go back to the balance \eqref{balanceDtu} with \eqref{I4I5bound} and \eqref{I6bound}
\begin{equation*}
\begin{aligned}
\frac{d}{dt}\Big(t^{2-\gamma-\varepsilon}\intrtwo \rho |D_t u|^2\hspace{0.05cm}dx\Big)+\frac{t^{2-\gamma-\varepsilon}}{8}\|\sqrt{\mu}\minspace\mathbb{D}D_t u\|_{L^2}^2 &\leq c(\delta)\minspace t^{1-\gamma-\varepsilon}\Big(\|D_t u\|_{L^2}^{2}+\|\nabla u\|_{L^2}^{4}\Big)\\
&\quad+c(\mumin,\delta)t^{2-\gamma-\varepsilon}\|\nabla u \|_{L^2}^2\| D_t u\|_{L^2}^2\\
&\quad+\frac{d}{dt}\Big(t^{2-\gamma-\varepsilon}\intrtwo \pai u_k \pk  u_i P\minspace dx\Big),
\end{aligned}
\end{equation*} 
and integrate in time to obtain
\begin{equation*}
\begin{aligned}
t^{2-\gamma-\varepsilon}\intrtwo \rho |D_t u|^2\hspace{0.05cm}dx+\frac{1}{8}\int_0^t \tau^{2-\gamma-\varepsilon}\|\sqrt{\mu}\minspace\mathbb{D}D_t u\|_{L^2}^2d\tau &\leq c(\rhomin,\mumin,\delta,\|\sqrt{\rho_0}u_0\|_{L^2}) \|u_0\|_{\dot{H}^{\gamma+\varepsilon}}^2\\
&\quad+c(\mumin,\delta)\int_0^t \tau^{2-\gamma-\varepsilon}\|\nabla u \|_{L^2}^2\| D_t u\|_{L^2}^2d\tau\\
&\quad+t^{2-\gamma-\varepsilon}\intrtwo \pai u_k \pk  u_i P\minspace dx,
\end{aligned}
\end{equation*} 
where we have used the previous energy estimate \eqref{balanceHg}.
Notice that the last term on the right-hand side is like $K_5$ \eqref{J1split} but with an additional factor of $t$ on the time weight. Therefore, from \eqref{K5bound} and \eqref{balanceNablau1}, we have the bound
\begin{equation*}
\begin{aligned}
t^{2-\gamma-\varepsilon}\intrtwo \pai u_k &\pk  u_i P\minspace dx\leq \frac{\minspace t^{2-\gamma-\varepsilon}}{2}\|\sqrt{\rho} D_t u\|_{L^2}^{2}+c(\delta)\minspace t^{2-\gamma-\varepsilon}\|\nabla u\|_{L^2}^{4}\\
&\leq \frac{\minspace t^{2-\gamma-\varepsilon}}{2}\|\sqrt{\rho} D_t u\|_{L^2}^{2}+c(\mumin,\delta,\|\sqrt{\rho_0}u_0\|_{L^2}) \|u_0\|_{\dot{H}^{\gamma+\varepsilon}}^2\frac{t^{2-\gamma-\varepsilon}}{(\max\{t^{1-\gamma-\varepsilon},t\})^2},
\end{aligned}
\end{equation*}
and thus
\begin{equation*}
\begin{aligned}
t^{2-\gamma-\varepsilon}\!\!\intrtwo \!\!\rho |D_t u|^2\hspace{0.05cm}dx\!+\!\frac{1}{4}\!\int_0^t\!\!\! \tau^{2-\gamma-\varepsilon}&\|\sqrt{\mu}\minspace\mathbb{D}D_t u\|_{L^2}^2d\tau \leq c(\mumin,\delta)\int_0^t \tau^{2-\gamma-\varepsilon}\|\nabla u \|_{L^2}^2\| D_t u\|_{L^2}^2d\tau\\
&+c(\rhomin,\mumin,\delta,\|\sqrt{\rho_0}u_0\|_{L^2}) \|u_0\|_{\dot{H}^{\gamma+\varepsilon}}^2(1\!+\!\min\{t^{\gamma+\varepsilon},t^{-\gamma-\varepsilon}\}).
\end{aligned}
\end{equation*} 
Gr\"onwall's lemma then allows us to conclude that
\begin{equation}\label{balanceH2}
\begin{aligned}
t^{2-\gamma-\varepsilon}\|D_t u\|_{L^2}^2+\int_0^t \tau^{2-\gamma-\varepsilon}\|\nabla D_t u\|_{L^2}^2d\tau &\leq c(\rhomin,\mumin,\delta,\|\sqrt{\rho_0}u_0\|_{L^2}) \|u_0\|_{\dot{H}^{\gamma+\varepsilon}}^2.
\end{aligned}
\end{equation} 
Repeating the steps but with the weight $t^{1+\frac{2}{r}-\varepsilon}$ in \eqref{balanceDtu} and using \eqref{balanceNablau} instead of  \eqref{balanceNablau1}, it is analogous to check that the following balance also holds
\begin{equation}
\begin{aligned}
t^{1+\frac{2}{r}-\varepsilon}\|D_t u\|_{L^2}^2+\int_0^t \tau^{1+\frac{2}{r}-\varepsilon}\|\nabla D_t u\|_{L^2}^2d\tau &\leq C,
\end{aligned}
\end{equation} 
and hence
\begin{equation}\label{balanceDtu_weight}
    \begin{aligned}
    \max\{t^{2-\gamma-\varepsilon},t^{1+\frac{2}{r}-\varepsilon}\}\|D_t u\|_{L^2}^2+\int_0^t  \max\{\tau^{2-\gamma-\varepsilon},\tau^{1+\frac{2}{r}-\varepsilon}\}\|\nabla D_t u\|_{L^2}^2d\tau &\leq C,
    \end{aligned}
\end{equation}
where $C=C(a^m, a^M, \|a_0-a^\infty\|_{L^2},\delta,\|u_0\|_{L^r}, \|u_0\|_{H^{\gamma+\varepsilon}})$, $a\equiv \rho, \mu$.
Next, to show that $u\in C(\mathbb{R}_+;H^{\gamma+\varepsilon})$, we write $(INS)$ as a forced heat equation,
\begin{equation*}
    \begin{aligned}
    u_t-\Delta u=-\mathbb{P}(\rho u\cdot\nabla u)+\mathbb{P}((1-\rho)u_t)+\mathbb{P}\nabla\cdot((\mu-1)\mathbb{D}u),
    \end{aligned}
\end{equation*}
and hence the velocity is given by
\begin{equation}\label{bootstrap}
    \begin{aligned}
    u&=e^{t\Delta}u_0+\int_0^te^{(t-\tau)\Delta}\big(-\mathbb{P}(\rho u\cdot\nabla u)+\mathbb{P}((1-\rho)u_t)+\mathbb{P}\nabla\cdot((\mu-1)\mathbb{D}u)\big)(\tau)d\tau\\
    &=v_1+v_2+v_3+v_4.
    \end{aligned}
\end{equation}
Ladyzhenskaya's inequality followed by \eqref{Lpv} gives that
\begin{equation*}
    \|\rho u\cdot \nabla u\|_{L^2}\leq c(\delta)\|u\|_{L^2}^{\frac12}\|\nabla u\|_{L^2}\|D_t u\|_{L^2}^{\frac12},
\end{equation*}
so the estimates \eqref{balanceHg} and \eqref{balanceH2} provides that
\begin{equation*}
    \|\mathbb{P}(\rho u\cdot \nabla u)\|_{L^2}\leq c(\rhomin,\mumin,\delta,\|u_0\|_{H^{\gamma+\varepsilon}})\minspace t^{-1+\frac{3}{4}(\gamma+\varepsilon)}.
\end{equation*}
Similarly, 
\begin{equation*}
    \|\mathbb{P}((1-\rho) u_t)\|_{L^2}\leq c(\rhomin,\mumin,\delta,\|u_0\|_{H^{\gamma+\varepsilon}})\minspace t^{-1+\frac{\gamma+\varepsilon}{2}}.
\end{equation*}
Hence, by Young's inequality for convolutions and the decay properties of the heat kernel, we obtain that
\begin{equation*}
\|v_1\|_{L^\infty_T(\dot{H}^{\gamma+\varepsilon})}\leq c\|u_0\|_{H^{\gamma+\varepsilon}},
\end{equation*}
\begin{equation*}
\begin{aligned}
\|v_2+v_3\|_{L^\infty_T(\dot{H}^{\gamma+\varepsilon})}&\leq c(\rhomin,\mumin,\delta,\|u_0\|_{H^{\gamma+\varepsilon}})\Big|\Big| \int_0^t (t-\tau)^{-\frac{\gamma+\varepsilon}{2}}(\tau^{-1+\frac{\gamma+\varepsilon}{2}}+\tau^{-1+\frac43(\gamma+\varepsilon)})\Big|\Big|_{L^\infty_T}\\
&\leq c(\rhomin,\mumin,\delta,\|u_0\|_{H^{\gamma+\varepsilon}}).
\end{aligned}
\end{equation*}
Estimate \eqref{balanceHg}, 
integration by parts and the arguments above give that
\begin{equation*}
\begin{aligned}
\|v_4\|_{L^\infty_T(\dot{H}^{\gamma+\varepsilon})}\leq c(\mumin,\delta,\|u_0\|_{H^{\gamma+\varepsilon}})\Big|\Big| \int_0^t (t-\tau)^{-\frac{1+\gamma+\varepsilon}{2}}\tau^{-\frac12+\frac{\gamma+\varepsilon}{2}}\Big|\Big|_{L^\infty_T}\leq c(\mumin,\delta,\|u_0\|_{H^{\gamma+\varepsilon}}).
\end{aligned}
\end{equation*}
Therefore, we conclude that
\begin{equation*}
\|u\|_{L^\infty(\mathbb{R}_+;H^{\gamma+\varepsilon})}\leq  c(\rhomin,\mumin,\delta,\|u_0\|_{H^{\gamma+\varepsilon}}).
\end{equation*}
The integration in time in \eqref{bootstrap} provides the continuity, following above estimates.

\subsection*{$L^p$-in-time estimates.}
We summarize here the $L^p(0,T)$ estimates that will be needed in Steps 4 and 5. It should be noticed the different constraints for short and long times. From the estimate $\eqref{balanceNablau}$, we get that for $2r/(2+r(1-\varepsilon))<p<2/(2-\gamma-\varepsilon)$,
\begin{equation*}
    \begin{aligned}
    \int_0^t\|D_tu\|_{L^2}^{p}d\tau&\leq \Big(\int_0^1\tau^{1-\gamma-\varepsilon}\|D_tu\|_{L^2}^2d\tau\Big)^{\frac{p}2}\Big(\int_0^1\tau^{-\frac{p(1-\gamma-\varepsilon)}{2}\frac{2}{2-p}}d\tau\Big)^{\frac{2-p}2}\\
    &\quad+\Big(\int_1^t\tau^{\frac{2}{r}-\varepsilon}\|D_tu\|_{L^2}^2d\tau\Big)^{\frac{p}2}\Big(\int_1^t\tau^{-\big(\frac{2}{r}-\varepsilon\big)\frac{p}{2-p}}d\tau\Big)^{\frac{2-p}2},
    \end{aligned}
\end{equation*}
and thus
\begin{equation}\label{Dtp}
\int_0^t\|D_tu\|_{L^2}^{p}d\tau \leq C(\rhomin, \mumin,\delta,\|u_0\|_{L^r}, \|u_0\|_{H^{\gamma+\varepsilon}}),\qquad \frac{2r}{2+r(1-\varepsilon)}<p<\frac{2}{2-\gamma-\varepsilon}.
\end{equation}
Similarly, we obtain from \eqref{balanceDtu_weight} that
\begin{equation}\label{DDtp}
        \int_0^t\|\nabla D_tu\|_{L^2}^{p}d\tau\leq c(\rhomin,\mumin,\delta,\|\sqrt{\rho_0}u_0\|_{L^2}) \|u_0\|_{\dot{H}^{\gamma+\varepsilon}}^p,\qquad \frac23<p<\frac{2}{3-\gamma-\varepsilon},
\end{equation}
and from Theorem \ref{HuangPaicuThm}, we have that
\begin{equation}\label{Dup}
    \|\nabla u(t)\|_{L^2}\leq C (t+e)^{-\frac{1}{r}+\frac{\varepsilon}{2}},\qquad \int_0^t\|\nabla u\|_{L^2}^p d\tau \leq C,\qquad \frac{2r}{2-r\varepsilon}<p\leq 2.
\end{equation}
Additionally, we have the following estimate
\begin{equation}\label{L1tlp}
    \begin{aligned}
    \int_0^t \|\nabla  u\|_{L^p} d\tau \leq C,\qquad \frac{2r}{2-r(1+\varepsilon)}<p<\infty.
    \end{aligned}
\end{equation}
In fact, by interpolation followed by \eqref{balanceNablau} and \eqref{balanceDtu_weight}, 
\begin{equation}\label{L1tlp_aux}
    \begin{aligned}
    \int_0^t \|\nabla  u\|_{L^p} d\tau \leq  c(\delta)\int_0^t \|\nabla u\|_{L^2}^{\frac2{p}}\|D_t u\|_{L^2}^{1-\frac{2}{p}}d\tau\leq C\Big(\int_0^1\frac{d\tau}{\tau^{1-\frac{\gamma+\varepsilon}2-\frac1{p}}}+\int_1^t\frac{d\tau}{\tau^{\frac{1}{r}+\frac12-\frac{\varepsilon}2-\frac{1}{p}}}\Big)\leq C.
    \end{aligned}
\end{equation}

For clarity in notation, in Steps $4$ and $5$ we will suppress the dependence on the initial data from the constants.

\vspace{0.3cm}

\noindent \textbf{Step 4:} $\nabla u\in L^1(0,T;L^\infty)$

\vspace{0.3cm}

We first use \eqref{elliptic} to write the gradient of $u$ as follows
\begin{equation}\label{aux3I7I8}
\begin{aligned}
\nabla u=-\nabla\Delta^{-1}\mathbb{P}\nabla\cdot \big(\big(\frac{\mu}{\bar{\mu}}-1\big)\mathbb{D}u\big)+\frac1{\bar{\mu}} \nabla \Delta^{-1}\mathbb{P} (\rho D_t u)=I_7+I_8.
\end{aligned}
\end{equation}
We proceed first with $I_8$. Sobolev embedding and \eqref{elliptic} provides that 
\begin{equation*}
    \begin{aligned}
    |I_8|&\leq  c\minspace\|\nabla \Delta^{-1}\mathbb{P} (\rho D_t u)\|_{W^{1,\frac{2}{1-\gamma}}}\\
    &\leq c(\| \nabla \Delta^{-1}\mathbb{P}\nabla\cdot (\mu \mathbb{D} u)\|_{L^{\frac{2}{1-\gamma}}}+\|\nabla\nabla \Delta^{-1}\mathbb{P} (\rho D_t u)\|_{L^{\frac{2}{1-\gamma}}}), 
    \end{aligned}
\end{equation*}
thus the boundedness of singular integrals in $L^p$ followed by \eqref{Lpv}, \eqref{L4}, gives that
\begin{equation}\label{I8bound}
    \begin{aligned}
    |I_8|&\leq  c(\delta)\minspace \Big(\|\nabla u\|_{L^2}^{1-\gamma}\|D_t u\|_{L^2}^{\gamma}+\|D_t u\|_{L^2}^{1-\gamma}\|\nabla D_t u\|_{L^2}^{\gamma}\Big).
    \end{aligned}
\end{equation}
H\"older's inequality with  $p=\frac{2}{(2-\gamma)(1-\gamma)}$ on the second yields that
\begin{equation*}\label{I8bound2}
    \begin{aligned}
    |I_8|&\leq  c(\delta)\minspace \Big(\|\nabla u\|_{L^2}^{1-\gamma}\|D_t u\|_{L^2}^{\gamma}+\|D_t u\|_{L^2}^{\frac{2}{2-\gamma}}+\|\nabla D_t u\|_{L^2}^{\frac{2}{3-\gamma}}\Big),
    \end{aligned}
\end{equation*}
so that \eqref{Dtp}-\eqref{DDtp} guarantees that the $L^1(0,T)$ norm of the last two the terms is bounded uniformly in $T$. The $L^1(0,T)$ of the first term in controlled by \eqref{L1tlp_aux}, and hence for any $t>0$,
\begin{equation}\label{I8L1bound}
    \int_0^t |I_8|d\tau\leq C.
\end{equation}
Next, the term $I_7$  is given in index notation by
\begin{equation*}
    \begin{aligned}
    \big(\nabla\Delta^{-1}\mathbb{P}\nabla\cdot\big(\frac{\mu}{\bar{\mu}}-1\big)\mathbb{D}u\big)_{i,j}&=\pai \Delta^{-1}\big(\delta_{j,k}-\pj \Delta^{-1}\pk\big)\pam\big(\big(\frac{\mu}{\bar{\mu}}-1\big)\mathbb{D}_{k,m} u\big) \\
    &= \big(R_i R_m \delta_{j,k}-R_i R_j R_k R_m\big)\big(\big(\frac{\mu}{\bar{\mu}}-1\big)\mathbb{D}_{k,m} u\big),
    \end{aligned}
\end{equation*}
where $\delta_{j,k}$ denotes the Kronecker delta, $R_i$ the Riesz transform and Einstein's summation convention is used.
Define the corresponding kernels
\begin{equation*}
    \begin{aligned}
    K_{i,j,k,m}(x)=\mathcal{F}^{-1}\big(\frac{\xi_i\xi_m}{|\xi|^2}\delta_{j,k}-\frac{\xi_i\xi_j\xi_k\xi_m}{|\xi|^4}\big)(x),
    \end{aligned}
\end{equation*}
so that 
\begin{equation*}
\begin{aligned}
    \big(\nabla\Delta^{-1}\mathbb{P}\nabla\cdot\big(\big(\frac{\mu}{\bar{\mu}}-1\big)\mathbb{D}u\big)\big)_{i,j}&=\int_{D(t)}K_{i,j,k,m}(x-y) \big(\frac{\muin(y)}{\bar{\mu}}-1\big)\mathbb{D}_{k,m}u(y)\minspace dy\\
    &\quad+\int_{\mathbb{R}^2\setminus D(t)}K_{i,j,k,m}(x-y) \big(\frac{\muout(y)}{\bar{\mu}}-1\big)\mathbb{D}_{k,m}u(y)\minspace dy.
\end{aligned}
\end{equation*}
In the following we shall use the notation
\begin{equation*}
    \nabla\Delta^{-1}\mathbb{P}\nabla\cdot\big(1_{D(t)}\mathbb{D}u\big)=\int_{D(t)}K(x-y)\cdot \mathbb{D}u(y)\minspace dy, 
\end{equation*}
so we have
\begin{equation}\label{I7split}
\begin{aligned}
    I_7=& \int_{D(t)}K(x-y) \cdot\big(\frac{\muin(y)}{\bar{\mu}}-1\big)\mathbb{D}u(y)\minspace dy
    \\
   &+\int_{\mathbb{R}^2\setminus D(t)}K(x-y)\cdot \big(\frac{\muout(y)}{\bar{\mu}}-1\big)\mathbb{D}u(y)\minspace dy= I_{7,1}+I_{7,2}.
\end{aligned}
\end{equation}
We focus on $I_{7,1}$ first. Consider first $x\in \overline{D(t)}$. 
In the following, whenever $x \in \partial D(t)$ we will define $f(x)$ as the limit from inside $D(t)$. Then, 
\begin{equation}\label{I71}
    \begin{aligned}
    I_{7,1}=&\frac{1}{\bar{\mu}}\int_{D(t)}K(x-y)\cdot(\muin(y)-\muin(x))\mathbb{D}u(y)dy\\
    &+\big(\frac{\muin(x)}{\bar{\mu}}-1\big)\int_{D(t)}K(x-y)\cdot\mathbb{D}u(y)dy=J_3+\big(\frac{\muin(x)}{\bar{\mu}}-1\big)J_4.
    \end{aligned}
\end{equation}
Therefore,
\begin{equation*}
    |J_3|\leq c \|\muin\|_{\dot{C}^{\lambda}(\overline{D(t)})}\|\mathbb{D}u\|_{L^\frac{2}{\lambda-\varepsilon}},
\end{equation*}
where 
\begin{equation*}
    \varepsilon<\lambda<\gamma.
\end{equation*}
By interpolation, 
\begin{equation*}
    |J_3|\leq c \|\muin\|_{\dot{C}^{\gamma}(\overline{D(t)})}^{\frac{\lambda}{{\gamma}}}\|\muin-\bar{\mu}\|_{L^\infty}^{1-\frac{\lambda}{{\gamma}}}\|\mathbb{D}u\|_{L^\frac{2}{\lambda-\varepsilon}},
\end{equation*}
and applying Young's inequality we obtain
\begin{equation*}
    \begin{aligned}
    \int_0^t|J_3|d\tau&\leq c\|\muin\|_{L^\infty_t\dot{C}^{\gamma}(\overline{D(t)})}^{\frac{\lambda}{{\gamma}}}\delta^{1-\frac{\lambda}{{\gamma}}}\int_0^t\|\mathbb{D}u\|_{L^\frac{2}{\lambda-\varepsilon}}d\tau\\
    &\leq \frac{\gamma-\lambda}{\gamma}\delta \|\muin\|_{L^\infty_t\dot{C}^{\gamma}(\overline{D(t)})}^{\frac{\lambda}{{\gamma-\lambda}}}+C\frac{\lambda}{\gamma}\Big(\int_0^t\|\mathbb{D}u\|_{L^\frac{2}{\lambda-\varepsilon}}d\tau\Big)^\frac{\gamma}{\lambda}.
    \end{aligned}
\end{equation*}
The last term is bounded in \eqref{L1tlp}, thus
\begin{equation}\label{J3bound}
    \int_0^t|J_3|d\tau\leq \delta \frac{\gamma-\lambda}{\gamma}\|\muin\|_{L^\infty_t\dot{C}^{\gamma}(\overline{D(t)})}^{\frac{\lambda}{{\gamma-\lambda}}}+C.
    \end{equation}
We proceed with $J_4$.
Without loss of generality, let $\varphi_0(x)$ be a defining function for the domain $D_0$, $D_0=\{x\in\R^2:\,\varphi_0(x)>0\}$ (see for example \cite{KRYZ2016}, page 119),
and $\varphi(x,t)=\varphi_0(X^{-1}(x,t))$ the corresponding defining function for $D(t)$. Define $\eta(t)$ as the following cut-off radius
\begin{equation}\label{constcut-off}
    \eta(t)=\min\{\Big(\frac{|\nabla \varphi|_{\inf}}{\|\nabla \varphi\|_{\dot{C}^\gamma}}\Big)^{\frac{1}{\gamma}},1\},
\end{equation}
where we use the notation
\begin{equation*}
    |\nabla\varphi|_{\inf}=\inf_{x\in\partial D}|\nabla\varphi(x)|.
\end{equation*}
Then, we split $J_4$ as follows
\begin{equation*}
\begin{aligned}
    J_4&= \int_{D(t)\cap\{|x-y|\leq \eta\}}\!\!\!K(x-y)\cdot\mathbb{D}u(y)+\int_{D(t)\cap\{|x-y|\geq \eta\}}\!\!\!K(x-y)\cdot\mathbb{D}u(y)= J_{4,1}+J_{4,2}.
    \end{aligned}
\end{equation*}
Since the kernel $K$ is even, the second term on the right below is bounded using \cite{BertozziConstantin1993} (see Geometric Lemma) as follows: 
\begin{equation*}
    \begin{aligned}
        |J_{4,1}|&\leq |\int_{D(t)\cap\{|x-y|\leq \eta\}}\!\!\!\!\!K(x-y)\cdot\big(\mathbb{D}u(y)-\mathbb{D}u(x)\big)dy|\!+\!|\mathbb{D}u(x)\cdot\int_{D(t)\cap\{|x-y|\leq \eta\}}\!\!\!\!\!K(x-y)dy|\\
        &\leq c\|\nabla u\|_{\dot{C}^\gamma(\overline{D(t)})}\int_0^\eta \frac{dr}{r^{1-\gamma}}+c(\gamma)\|\nabla u\|_{L^\infty}\leq c\minspace \|\nabla u\|_{\dot{C}^\gamma(\overline{D(t)})}+c(\gamma)\|\nabla u\|_{L^\infty}.
    \end{aligned}
\end{equation*}
The term $J_{4,2}$ is bounded by
\begin{equation*}
    \begin{aligned}
        |J_{4,2}|\leq c\Big(\int_{\eta(t)}^\infty \frac{dr}{r^{\frac{2+\gamma}{2-\gamma}}}\Big)^{\frac{2-\gamma}{2}}\|\nabla u\|_{L^\frac{2}{\gamma}}=c\minspace\eta^{-\gamma} \|\nabla u\|_{L^\frac{2}{\gamma}}.
    \end{aligned}
\end{equation*}
Thus, joining the bounds for $J_{4,1}$ and $J_{4,2}$ we find that
\begin{equation*}
    \begin{aligned}
        |J_4|\leq  c(\gamma)\|\nabla u\|_{L^\infty}+c\minspace \|\nabla u\|_{\dot{C}^\gamma(\overline{D(t)})}+c\minspace\eta^{-\gamma} \|\nabla u\|_{L^\frac{2}{\gamma}},
    \end{aligned}
\end{equation*}
and going back to \eqref{I71} with the bound \eqref{J3bound} we have that
\begin{equation}\label{I71bound}
\begin{aligned}
    \int_0^t |I_{7,1}|d\tau\leq C +\delta\big(& \frac{\gamma-\lambda}{\gamma}\|\muin\|_{L^\infty_t\dot{C}^{\gamma}(\overline{D(t)})}^{\frac{\lambda}{{\gamma-\lambda}}}+c(\gamma)\int_0^t\|\nabla u\|_{L^\infty}d\tau\\
    &+c\int_0^t\|\nabla u\|_{\dot{C}^\gamma(\overline{D(t)})}d\tau+c\int_0^t\eta^{-\gamma} \|\nabla u\|_{L^\frac{2}{\gamma}}d\tau\big).
\end{aligned}
\end{equation}
Notice that if $x\notin \overline{D(t)}$, we can define $\tilde{x}$ as a point on the boundary with minimum distance to $x$, $\tilde{x}=\arg d(x,\partial D(t))\in \partial D(t)$. Then, by adding and subtracting
$$\muin(\tilde{x})=\lim_{x\to \tilde{x},\, x\in D(t)}\muin(x)$$
in \eqref{I71} instead of $\muin(x)$,
$$ \mathbb{D}u(\tilde{x})=\lim_{x\to \tilde{x},\, x\in D(t)}\mathbb{D}u(x)$$
for the $J_{4,1}$ term instead of $\mathbb{D}u(x)$, and using the triangle inequality, we obtain the same bounds for $J_3$ and $J_{4,1}$. Since the bound for $J_{4,2}$ holds equally for $x\notin \overline{D(t)}$, we have that the estimate \eqref{I71bound} holds for any $x\in\mathbb{R}^2$.

The term $I_{7,2}$ \eqref{I7split} is decomposed further, $I_{7,2}=J_5+J_6$, with:
$$
J_5=\int_{(\mathbb{R}^2\setminus D(t))\cap\{|x-y|>1\}}K(x-y)\cdot \big(\frac{\muout(y)}{\bar{\mu}}-1\big)\mathbb{D}u(y)\minspace dy,
$$
and
$$
J_6=\int_{(\mathbb{R}^2\setminus D(t))\cap\{|x-y|<1\}}K(x-y)\cdot \big(\frac{\muout(y)}{\bar{\mu}}-1\big)\mathbb{D}u(y)\minspace dy.
$$
At this point, it is direct to bound $J_5:$
$$
\int_0^t|J_5|d\tau\leq c\|\frac{\muout}{\bar{\mu}}-1\|_{L^\infty}\int_0^t\|\nabla u\|_{L^{\frac2\gamma}}d\tau\leq c\delta \int_0^t \eta^{-\gamma}\|\nabla u\|_{L^{\frac2\gamma}}d\tau.
$$
The term $J_6$ is handled as $I_{7,1}$ \eqref{I7split} to get the estimate
\begin{equation*}
\begin{aligned}
    \int_0^t |I_{7,2}|d\tau\leq C +\delta\big(& \frac{\gamma-\lambda}{\gamma}\|\muout\|_{L^\infty_t\dot{C}^{\gamma}(\mathbb{R}^2\setminus D(t))}^{\frac{\lambda}{{\gamma-\lambda}}}+c(\gamma)\int_0^t\|\nabla u\|_{L^\infty}d\tau\\
    &+c\int_0^t\|\nabla u\|_{\dot{C}^\gamma(\mathbb{R}^2\setminus D(t))}d\tau+c\int_0^t\eta^{-\gamma} \|\nabla u\|_{L^\frac{2}{\gamma}}d\tau\big).
\end{aligned}
\end{equation*}
Going back to \eqref{I7split}, above estimate and \eqref{I71bound} allows to get
\begin{equation*}
\begin{aligned}
    \int_0^t |I_{7}|d\tau\leq C +\delta\big(& \frac{\gamma-\lambda}{\gamma}\|\mu\|_{L^\infty_t\dot{C}^\gamma(\overline{D(t)})\cap \dot{C}^\gamma(\mathbb{R}^2\setminus D(t))}^{\frac{\lambda}{{\gamma-\lambda}}}+c(\gamma)\int_0^t\|\nabla u\|_{L^\infty}d\tau\\
    &+c\int_0^t\|\nabla u\|_{\dot{C}^\gamma(\overline{D(t)})\cap \dot{C}^\gamma(\mathbb{R}^2\setminus D(t))}d\tau+c\int_0^t\eta^{-\gamma} \|\nabla u\|_{L^\frac{2}{\gamma}}d\tau\big).
\end{aligned}
\end{equation*}
Splitting \eqref{aux3I7I8}, above estimate together with \eqref{I8L1bound} provide
\begin{equation}\label{gradu_inf}
    \begin{aligned}
\int_0^T\|\nabla u\|_{L^\infty}dt\leq C+
        \frac{c\delta}{1\!-\!c(\gamma)\delta}\Big(&\|\mu\|_{L^\infty_T\dot{C}^\gamma(\overline{D(t)})\cap \dot{C}^\gamma(\mathbb{R}^2\setminus D(t))}^{\frac{\lambda}{{\gamma-\lambda}}}+
       \int_0^T\!\!\eta(t)^{-\gamma} \|\nabla u\|_{L^\frac{2}{\gamma}}dt \\
        &+\int_0^T\|\nabla u\|_{\dot{C}^\gamma(\overline{D(t)})\cap \dot{C}^\gamma(\mathbb{R}^2\setminus D(t))}dt\Big).
    \end{aligned}
\end{equation}

\vspace{0.3cm}
\noindent \textbf{Step 5:} $\nabla u\in L^1(0,T;\dot{C}^\gamma(\overline{D(t)})\cap \dot{C}^\gamma(\mathbb{R}^2\setminus D(t)))$
\vspace{0.3cm}

Recalling the expression \eqref{aux3I7I8} for $\nabla u$, we write
\begin{equation}\label{aux4}
    \begin{aligned}
        \nabla u(x+h)-\nabla u(x)&= I_{9,1}+I_{9,2}+I_{10},
    \end{aligned}
\end{equation}
where
$$
I_{9,1}=\int_{D(t)}\big(K(x\!+\!h\!-\!y)\!-\!K(x\!-\!y)\big)\big(\frac{\muin(y)}{\bar{\mu}}-1\big)\mathbb{D}u(y)dy,$$
$$
I_{9,2}=\int_{\mathbb{R}^2\setminus D(t)}\!\!\big(K(x\!+\!h\!-\!y)\!-\!K(x\!-\!y)\big)\big(\frac{\muout(y)}{\bar{\mu}}-1\big)\mathbb{D}u(y)dy,$$
and
$$
I_{10}=\frac1{\bar{\mu}}(\nabla\Delta^{-1}\mathbb{P}\big(\rho D_t u\big)(x+h)-\nabla\Delta^{-1}\mathbb{P}\big(\rho D_t u\big)(x)).
$$
We deal first with the term $I_{10}$. Classical Sobolev embedding together with Gagliardo-Nirenberg inequality \eqref{L4} gives that
\begin{equation}\label{I10}
    \begin{aligned}
    |I_{10}|&\leq c \|\nabla\Delta^{-1}\mathbb{P}\big(\rho D_t u\big)\|_{\dot{W}^{1,\frac{2}{1-\gamma}}}|h|^\gamma\leq c\|D_tu\|_{L^{\frac{2}{1-\gamma}}}|h|^\gamma\leq c \|D_tu\|_{L^2}^{1-\gamma}\|\nabla D_tu\|_{L^2}^{\gamma}|h|^\gamma,
    \end{aligned}
\end{equation}
so repeating the steps for \eqref{I8bound} we conclude that $I_{10}$ is uniformly bounded in $L^1(0,T)$.

Next we deal with $I_{9,1}$ and $I_{9,2}$. Assume that $x$ and $x+h$ belong to $\overline{D(t)}$. 
We proceed first with the term $I_{9,1}$. We decompose $I_{9,1}$ as follows:
\begin{equation}\label{I9}
    \begin{aligned}
        I_{9,1}=L_1+L_2+L_3+L_4+L_5,
    \end{aligned}
\end{equation}
with
\begin{equation*}
    \begin{aligned}
        L_1&=\int_{D(t)\cap \{|x-y|<2|h|\}}K(x+h-y)
        \Big(
        \big(\big(\frac{\muin}{\bar{\mu}}-1\big)\mathbb{D}u\big)(y)-
        \big(\big(\frac{\muin}{\bar{\mu}}-1\big)\mathbb{D}u\big)(x+h)
        \Big)dy,
 \end{aligned}
\end{equation*}
 
\begin{equation*}
    \begin{aligned}        L_2&=\int_{D(t)\cap \{|x-y|<2|h|\}}K(x-y)\Big(\big(\big(\frac{\muin}{\bar{\mu}}-1\big)\mathbb{D}u\big)(x)-\big(\big(\frac{\muin}{\bar{\mu}}-1\big)\mathbb{D}u\big)(y)\Big)dy,
   \end{aligned}
\end{equation*} 
\begin{equation*}
    \begin{aligned}      
        L_3&=\int_{D(t)\cap \{|x-y|<2|h|\}}K(x-y)\Big(\big(\big(\frac{\muin}{\bar{\mu}}-1\big)\mathbb{D}u\big)(x+h)-\big(\big(\frac{\muin}{\bar{\mu}}-1\big)\mathbb{D}u\big)(x)\Big)dy,
        \end{aligned}
\end{equation*}
\begin{equation*}
    \begin{aligned}  
        L_4&=\!\int_{D(t)\cap \{|x-y|\geq 2|h|\}}\!\!\!\!\!\!\!\!(K(x\!+\!h\!-\!y)\!-\!K(x\!-\!y))\Big(        \big(\big(\frac{\muin}{\bar{\mu}}\!-\!1\big)\mathbb{D}u\big)(y)\!-\!
        \big(\big(\frac{\muin}{\bar{\mu}}\!-\!1\big)\mathbb{D}u\big)(x\!+\!h)\Big)dy,
        \end{aligned}
\end{equation*}
and
\begin{equation*}
    \begin{aligned}   L_5&=\big(\frac{\muin(x+h)}{\bar{\mu}}-1\big)\mathbb{D}u(x+h)\int_{D(t)}\big(K(x+h-y)-K(x-y)\big)dy.
    \end{aligned}
\end{equation*}
The terms $L_1$ and $L_2$ are directly bounded by
\begin{equation*}
    |L_1|+|L_2|\leq c\|\big(\frac{\muin}{\bar{\mu}}-1\big)\mathbb{D}u\|_{\dot{C}^\gamma(\overline{D(t)})}|h|^\gamma.
\end{equation*}
Choosing $2|h|\leq \min_{t\in[0,T]}\eta(t)$, we have that (see \cite{BertozziConstantin1993})
 \begin{equation*}
    |L_3|\leq c\|\big(\frac{\muin}{\bar{\mu}}-1\big)\mathbb{D}u\|_{\dot{C}^\gamma(\overline{D(t)})}|h|^\gamma.
\end{equation*}
 By the mean value theorem we also have that
 \begin{equation*}
     |L_4|\leq c\|\big(\frac{\muin}{\bar{\mu}}-1\big)\mathbb{D}u\|_{\dot{C}^\gamma(\overline{D(t)})}|h|^\gamma.
 \end{equation*}
 The term $L_5$ is more singular and we have to use contour dynamics to control it. According to Theorem \ref{SingIntHolder},
\begin{equation*}
    \begin{aligned}
        |L_5|\leq c\|\big(\frac{\muin}{\bar{\mu}}-1\big)\mathbb{D}u\|_{L^\infty(\overline{D(t)})}\mathcal{P}(\|D(t)\|_{\text{Lip}}+\|D(t)\|_{*}) \|D(t)\|_{\dot{C}^{1,\gamma}}|h|^\gamma,
    \end{aligned}
\end{equation*} 
where $\mathcal{P}$ is a polynomial.
Let $z_0(\alpha)$ be a $C^{1,\gamma}$ parametrization of the initial domain $D_0$, so that its evolution via the particle trajectories \eqref{particle_trajectories}
\begin{equation*}
    z(\alpha,t)=X(z_0(\alpha),t),
\end{equation*}
gives the parametrization of $D(t)$. Then, 
since
\begin{equation*}
    \|\pa z\|_{\dot{C}^\gamma}\leq (\|\pa z_0\|_{C^\gamma}+1)^{1+\gamma}\|\nabla X\|_{C^\gamma(\overline{D(t)})},
\end{equation*}
we obtain that
\begin{equation*}
   \begin{aligned}
        |L_5|\leq c\|\big(\frac{\muin}{\bar{\mu}}-1\big)\mathbb{D}u\|_{L^\infty(\overline{D(t)})} \mathcal{P}(\|\nabla \varphi\|_{L^\infty}+|\nabla \varphi|_{\inf}^{-1})\|\nabla X\|_{C^\gamma(\overline{D(t)})}|h|^\gamma,
    \end{aligned}
\end{equation*}
where the constant $c$ only depends on the initial domain. Then, we notice that
\begin{equation*}
    \|\big(\frac{\muin}{\bar{\mu}}-1\big)\mathbb{D}u\|_{\dot{C}^\gamma(\overline{D(t)})}\leq c\|\muin\|_{\dot{C}^\gamma(\overline{D(t)})}\|\nabla u\|_{L^\infty}+c\delta\|\nabla u\|_{\dot{C}^\gamma(\overline{D(t)})}.
\end{equation*}
The same approach is done to control $I_{9,2}$ but using the decomposition $\tilde{L}_1-\tilde{L}_5$ where the domain $D(t)$ is replaced by $\mathbb{R}^2\setminus D(t)$ and using  $((\frac{\muout}{\bar{\mu}}-1)\mathbb{D}u)(\tilde{x})$ and $((\frac{\muout}{\bar{\mu}}-1)\mathbb{D}u)(\tilde{x}_h)$ instead, where $\tilde{x}_h=\arg d(x+h,\partial D(t))$. 
In fact, since for $y\in\mathbb{R}^2\setminus D(t)$ one has that $|y-\tilde{x}|\leq 2|y-x|$, $|y-\tilde{x}_h|\leq 2|x+h-y|$, we obtain
\begin{equation*}
    |\tilde{L}_1|+|\tilde{L}_2|\leq c\|\big(\frac{\muout}{\bar{\mu}}-1\big)\mathbb{D}u\|_{\dot{C}^\gamma(\mathbb{R}^2\setminus D(t))}|h|^\gamma.
\end{equation*}
The term $\tilde{L}^4$ is done analogously. For $\tilde{L}^3$, notice that the only nontrivial case happens when $d(x,\partial D(t))<2|h|$ (otherwise $\tilde{L}_3=0$). Also, since the domain is restricted to a ball, the integral of the kernel is  bounded in the same way (\cite{BertozziConstantin1993}). We thus have that
\begin{equation*}
    |\tilde{L}_3|\leq c\|(\frac{\muout}{\bar{\mu}}-1)\mathbb{D}u\|_{\dot{C}^\gamma(\mathbb{R}^2\setminus D(t))}|\tilde{x}_h-\tilde{x}|^\gamma\leq c\|(\frac{\muout}{\bar{\mu}}-1)\mathbb{D}u\|_{\dot{C}^\gamma(\mathbb{R}^2\setminus D(t))}|h|^\gamma,
\end{equation*}
where in the last step we use that $|\tilde{x}_h-\tilde{x}|\leq |\tilde{x}-(x+h)|+|(x+h)-\tilde{x}_h|\leq 6|h|$. The term $\tilde{L}_5$ follows directly from Theorem \ref{SingIntHolder},
\begin{equation*}
    \begin{aligned}
        |\tilde{L}_5|\leq c\|\big(\frac{\muout}{\bar{\mu}}-1\big)\mathbb{D}u\|_{L^\infty(\mathbb{R}^2\setminus D(t))}\mathcal{P}(\|D(t)\|_{\text{Lip}}+\|D(t)\|_{*}) \|D(t)\|_{\dot{C}^{1,\gamma}}|h|^\gamma.
    \end{aligned}
\end{equation*} 
The case $x, x+h\in\mathbb{R}^2\setminus D(t)$ is done analogously. It yields all the desired estimates to control the H\"older norm for $\nabla u$. In order to gather them all, we will denote $$\dot{C}^\gamma_D=\dot{C}^\gamma(\overline{D(t)})\cap\dot{C}^\gamma(\mathbb{R}^2\setminus D(t))$$
for clarity in notation, and analogously for $C^\gamma_D$, to be used in the following.
Therefore, from splitting \eqref{aux4} and above estimates it is possible to get
\begin{equation}\label{nablaucgamma}
\begin{aligned}
\|\nabla u\|_{\dot{C}^\gamma_D}\leq& \frac{c\delta}{1-c\delta}\|\nabla u\|_{L^\infty} \mathcal{P}(\|\nabla \varphi\|_{L^\infty}+|\nabla \varphi|_{\inf}^{-1})\|\nabla X\|_{C^\gamma_D}\\
&+\frac{c}{1\!-\!c\delta}\|\mu\|_{\dot{C}^\gamma_D}\|\nabla u\|_{L^\infty}+\frac{c}{1\!-\!c\delta}\Big( \|D_t u\|_{L^2}^{\frac{2}{2-\gamma}}\!+\!\|\nabla D_t u\|_{L^2}^{\frac{2}{3-\gamma}}\Big).
\end{aligned}
\end{equation}

\vspace{0.3cm}
\noindent \textbf{Step 6:} Closing all estimates. \vspace{0.3cm}

We now introduce the bound \eqref{nablaucgamma} back in \eqref{gradu_inf} to obtain
\begin{equation}\label{aux7}
    \begin{aligned}
     \int_0^t\|\nabla u\|_{L^\infty}d\tau&\leq c\minspace\delta^2\int_0^t \|\nabla u\|_{L^\infty}\mathcal{P}(\|\nabla \varphi\|_{L^\infty}\!+\!|\nabla \varphi|_{\inf}^{-1})\|\nabla X\|_{C^\gamma_D}d\tau\\
     &\quad+c\minspace\delta\int_0^t \Big(\|D_t u\|_{L^2}^{\frac{2}{2-\gamma}}+\|\nabla D_t u\|_{L^2}^{\frac{2}{3-\gamma}}\Big)d\tau+c\minspace\delta e^{c\int_0^t\|\nabla u\|_{L^\infty}d\tau}\\
     &\quad+c\minspace\delta\int_0^t\frac{\|\nabla \varphi\|_{\dot{C}^\gamma}}{|\nabla\varphi|_{\inf}}\|\nabla u\|_{L^{\frac{2}{\gamma}}}d\tau+C.
    \end{aligned}
\end{equation}
Above we have used the definition of $\eta$ \eqref{constcut-off} and that $\mu$ is transported by the flow:
$$
\|\mu\|_{\dot{C}_D^\gamma}\leq \|\mu_0\|_{\dot{C}_D^\gamma}e^{\gamma\int_0^t\|\nabla u\|_{L^\infty}d\tau}.
$$

We denote
\begin{equation*}
    y(t)=\int_0^t\|\nabla u\|_{L^\infty}d\tau,
\end{equation*}
and notice that
\begin{equation}\label{expLinf}
    \begin{aligned}
    \|\nabla X\|_{L^\infty}&\leq \|\nabla X_0\|_{L^\infty}e^{y(t)},\\
    \|\nabla \varphi\|_{L^\infty}&\leq \|\nabla \varphi_0\|_{L^\infty}e^{y(t)},\\
    |\nabla \varphi|_{\inf}&\geq |\nabla \varphi_0|_{\inf}e^{-y(t)}.
    \end{aligned}
\end{equation}
Furthermore, recalling the definition $\varphi(x,t)=\varphi_0(X^{-1}(x,t))$, we have that
\begin{equation*}
\begin{aligned}
    \nabla \varphi(X(x,t),t)&=(\nabla X(x,t))^{-1}\nabla \varphi_0(x),
\end{aligned}
\end{equation*}

\begin{equation*}
\begin{aligned}
    \frac{\nabla\varphi(X(x,t),t)-\nabla \varphi(X(y,t),t)}{|X(x,t)-X(y,t)|^\gamma}&=\frac{(\nabla X(x,t))^{-1}-(\nabla X(y,t))^{-1}}{|x-y|^\gamma}\nabla\varphi_0(x)\Big(\frac{|x-y|}{|X(x,t)\!-\!X(y,t)|}\Big)^\gamma\\
    &\quad+(\nabla X(y,t))^{-1}\frac{\nabla \varphi_0(x)-\nabla \varphi_0(y)}{|x-y|^\gamma}\Big(\frac{|x-y|}{|X(x,t)-X(y,t)|}\Big)^\gamma,
\end{aligned}
\end{equation*}
and hence, 
\begin{equation*}
    \begin{aligned}
    \|\nabla \varphi\|_{\dot{C}^{\gamma}}&\leq \big(\|(\nabla X)^{-1}\|_{L^\infty}^2\|\nabla X\|_{\dot{C}^{\gamma}_D}\|\nabla \varphi_0\|_{L^\infty}+\|(\nabla X)^{-1}\|_{L^\infty}\|\nabla \varphi_0\|_{\dot{C}^\gamma}\big)\|\nabla X^{-1}\|_{L^\infty}^\gamma.
    \end{aligned}
\end{equation*}
Using \eqref{expLinf}, we conclude that
\begin{equation}\label{expLinf2}
    \|\nabla \varphi\|_{\dot{C}^{\gamma}}\leq c\big(\|\nabla X\|_{\dot{C}^{\gamma}_D}+1\big)e^{c(\gamma)y(t)}.
\end{equation}

Next, we propagate further regularity
\begin{equation*}
\begin{aligned}
\frac{d}{dt}\|\nabla X\|_{L^\infty}&\leq \|\nabla u\|_{L^\infty}\|\nabla X\|_{L^\infty},\\
    \frac{d}{dt}\|\nabla X\|_{\dot{C}^\gamma_D}&\leq \|\nabla u\|_{L^\infty}\|\nabla X\|_{\dot{C}^\gamma_D}+\|\nabla X\|_{L^\infty}^{1+\gamma}\|\nabla u\|_{\dot{C}^\gamma_D},
\end{aligned}
\end{equation*}
and substitute the estimate \eqref{nablaucgamma} to obtain
\begin{equation*}
\begin{aligned}
\frac{d}{dt}\|\nabla X\|_{C^\gamma_D}\leq& \|\nabla u\|_{L^\infty}\|\nabla X\|_{C^\gamma_D}\\
&+ \delta\|\nabla u\|_{L^\infty}\mathcal{P}(\|\nabla \varphi\|_{L^\infty}\!\!+\!\!|\nabla \varphi|_{\inf}^{-1})\|\nabla X\|_{L^\infty}^{1+\gamma}\|\nabla X\|_{C^\gamma_D}\\
&+c\|\mu\|_{\dot{C}^\gamma_D}\|\nabla u\|_{L^\infty}\|\nabla X\|_{L^\infty}^{1+\gamma}+c\|\nabla X\|_{L^\infty}^{1+\gamma}\Big(\|D_t u\|_{L^2}^{\frac{2}{2-\gamma}}+\|\nabla D_t u\|_{L^2}^{\frac{2}{3-\gamma}}\Big).
\end{aligned}
\end{equation*}
We denote
\begin{equation}\label{abcd}
    \begin{aligned}
    x(t)&=\|\nabla X\|_{C^\gamma_D},\\
    a(t)&=\|\nabla u\|_{L^\infty}+\delta\|\nabla u\|_{L^\infty}\mathcal{P}(\|\nabla \varphi\|_{L^\infty}\!+\!|\nabla \varphi|_{\inf}^{-1})\|\nabla X\|_{L^\infty}^{1+\gamma},\\
    d(t)&=c\|\mu\|_{\dot{C}^\gamma_D}\|\nabla u\|_{L^\infty}\|\nabla X\|_{L^\infty}^{1+\gamma}+c\|\nabla X\|_{L^\infty}^{1+\gamma}\Big(\|D_t u\|_{L^2}^{\frac{2}{2-\gamma}}+\|\nabla D_t u\|_{L^2}^{\frac{2}{3-\gamma}}\Big),
    \end{aligned}
\end{equation}
so that the above inequality rewrites as follows
\begin{equation*}
    \dot{x}(t)\leq a(t)x(t)+d(t),
\end{equation*}
and after integration
\begin{equation*}
    x(t)\leq g(t)+\int_0^t a(s)x(s)ds,
\end{equation*}
where 
\begin{equation}\label{gfun}
    g(t)=x(0)+\int_0^t d(s)ds.
\end{equation}
Hence, applying Gr\"onwall's lemma, we find that
\begin{equation}\label{aux6}
    x(t)\leq  \|g\|_{L^\infty_T}e^{\int_0^t a(s)ds}.
\end{equation}

Next, the terms $a(t)$ and $g(t)$ are estimated using \eqref{expLinf} and \eqref{Dtp}-\eqref{DDtp},
\begin{equation*}
    \begin{aligned}
        \|g\|_{L^\infty_T}&\leq c\|\mu\|_{L^\infty_T(\dot{C}^\gamma_D)}y(T)e^{(1+\gamma)y(T)}+C(1+e^{(1+\gamma)y(T)})\leq  C(1+y(T))e^{c\mminspace y(T)}\leq Ce^{c\mminspace y(T)},\\
        \int_0^t a(s)ds&\leq y(t)+c\delta y(t)\minspace e^{c\mminspace y(t)}.
    \end{aligned}
\end{equation*}
Thus, by \eqref{aux6}, we have that
\begin{equation}\label{aux8}
\begin{aligned}
 x(t)\leq Ce^{c\mminspace y(T)}e^{c\delta y(t)e^{c y(t)}}.
\end{aligned}
\end{equation}
We introduce the bound \eqref{aux8} in \eqref{aux7}, together with \eqref{Dtp}-\eqref{L1tlp} and\eqref{expLinf}, to get that
\begin{equation*}
    \begin{aligned}
    y(t)&\leq c\mminspace\delta^2y(T)e^{c\mminspace y(T)}e^{c\delta y(t)e^{c y(t)}}+\delta Ce^{y(T)}+C(1+\delta),
    \end{aligned}
\end{equation*}
that is
\begin{equation*}
    y(T)\leq \delta C_1e^{C_2y(T)}e^{C_3\delta y(T)e^{C_4 y(T)}}+C_5.
\end{equation*}
Assume that 
\begin{equation}\label{eps_condition}
    \delta\leq \min\{\frac{e^{-2C_4C_5}}{2C_3C_5},\frac{C_5e^{-2C_2C_5-1}}{2C_1}\}. 
\end{equation}
Then, we proceed by contradiction. If there exists a first time $T$ such that $y(T)=2C_5$, we would then have that
\begin{equation*}
    \begin{aligned}
    2C_5\leq C_5\Big(\frac{C_1}{C_5}\delta e^{ 2C_2C_5+1}+1\Big)\leq \frac32 C_5.
    \end{aligned}
\end{equation*}
Therefore we conclude that there exists an $\delta>0$ satisfying \eqref{eps_condition} such that for any $T>0$  it holds that
\begin{equation*}
    y(T)=\int_0^T\|\nabla u\|_{L^\infty}dt< 2C_5,
\end{equation*}
and hence, from \eqref{expLinf}, \eqref{aux8} and  \eqref{nablaucgamma}, for all $t>0$, 
\begin{equation*}
\begin{aligned}
\eta(t)^\gamma&\geq C e^{-c\mminspace C_1}>0,\\
\|\nabla X\|_{C^\gamma_D}&\leq C,\\
\int_0^t\|\nabla u\|_{\dot{C}^\gamma_D}d\tau&\leq C.
\end{aligned}
\end{equation*}

\qed

\vspace{0.2cm}

\subsection{Uniqueness:} As in the case of constant viscosity \cite{DanchinMucha2012, DanchinMucha2013, DanchinMucha2019}, the uniqueness of solutions is proved in Lagrangian variables. This is due to the low regularity of the density and viscosity, produced by their jumps across the interface. 

We denote $(\rho_0,\mu_0,v,Q)$ the solution to $(INS)$ in Lagrangian coordinates, \begin{equation*}
    \begin{aligned}
    \rho_0(y)=\rho(X(y,t),t), \quad &\mu_0(y)=\mu(X(y,t),t), \\ v(y,t)=u(X(y,t),t), \quad &Q(y,t)=P(X(y,t),t),
    \end{aligned}
\end{equation*}
where $X$ is the flow defined by \eqref{particle_trajectories}. Note that
\begin{equation*}
    \nabla X(y,t)=\mathbb{I}_2+\int_0^t \nabla v(y,\tau)d\tau,
\end{equation*}
and denote
\begin{equation*}
    A(t)=(\nabla X(\cdot,t))^{-1}.
\end{equation*}
Then, in Lagrangian variables the operators $\nabla$, $\nabla\cdot$ are given as follows. If we denote $\tilde{f}(y,t)=f(X(y,t),t)$, then
\begin{equation*}
\begin{aligned}
    \nabla_u \tilde{f}(y,t)&:= (\nabla f)(X(y,t),t)=A^*\nabla \tilde{f}(y,t),\\
    \nabla_u\cdot \tilde{f}(y,t)&:=(\nabla\cdot f)(X(y,t),t)=\nabla\cdot(A \tilde{f}(y,t)),
\end{aligned}
\end{equation*}
and furthermore, since $\det{A(t)}=1$ due to the incompressibility condition, the following identity holds for vector fields (see e.g. \cite{DanchinMucha2012})
\begin{equation}\label{iden_div}
    \nabla\cdot(A\tilde{f})=A^*:\nabla \tilde{f}.
\end{equation}
Hence, $(INS)$ in $y\in\mathbb{R}^2$, $0<t<T$, rewrites as follows
\begin{equation*}
    \begin{aligned}
    \rho_0 \partial_tv&=\nabla_u\cdot(\mu_0\mathbb{D}_u v-Q\mathbb{I}_2),\\
    \nabla_u\cdot v&=0.
    \end{aligned}
\end{equation*}
The equivalence of these formulations is guaranteed assuming that
\begin{equation}\label{small_lips}
    \int_0^T \|\nabla v\|_{L^\infty}d\tau \leq c<1.
\end{equation}
In that case, one can write that
\begin{equation}\label{small_lips2}
    A(t)=\sum_{j=0}^\infty (-1)^j\Big(\int_0^t\nabla v(\cdot,\tau)d\tau\Big)^j.
\end{equation}
Let $(\rho^1,\mu^1,u^1,P^1,X^1)$, $(\rho^2,\mu^2,u^2,P^2,X^2)$ be two solutions as in Theorem \ref{Case1} for the same initial data. In Lagrangian coordinates, we denote their difference by
\begin{equation*}
    \delta  v= v^2- v^1,\quad     \delta  Q= Q^2- Q^1,
\end{equation*}
so that
\begin{equation*}
    \rho_0\delta v_t=\nabla_{u^1}\cdot(\mu_0\mathbb{D}_{u^1}\delta v-\delta Q\mathbb{I}_2)+\nabla_{u^2}\cdot(\mu_0\mathbb{D}_{u^2} v^2- Q^2\mathbb{I}_2)-\nabla_{u^1}\cdot(\mu_0\mathbb{D}_{u^1} v^2- Q^2\mathbb{I}_2),
\end{equation*}
\begin{equation*}
    \nabla_{u^1}\cdot\delta v=(\nabla_{u^1}-\nabla_{u^2})\cdot v^2,
\end{equation*}
\begin{equation*}
    \delta v|_{t=0}=0.
\end{equation*}    
We will now prove that for $T>0$ small enough, 
\begin{equation}\label{aux_uniq}
\int_0^T\int_{\mathbb{R}^2}|\nabla \delta  v|^2dy dt=0.
\end{equation}
First, let $\delta v=w+z$, where $w$ is assumed to be the solution to the equation
\begin{equation}\label{w_equation}
    \nabla_{u^1}\cdot w=(\nabla_{u^1}-\nabla_{u^2})\cdot  v^2=\nabla \cdot(\delta A v^2).
\end{equation}
Then, the equations for $z$ become
\begin{equation*}
    \begin{aligned}
    \rho_0 z_t-\nabla_{u^1}\cdot(\mu_0\mathbb{D}_{u^1}z)=&-\nabla_{u^1}\delta  Q+\nabla_{u^2}\cdot(\mu_0\mathbb{D}_{u^2} v^2-Q^2\mathbb{I}_2)-\nabla_{u^1}\cdot(\mu_0\mathbb{D}_{u^1} v^2-Q^2\mathbb{I}_2)\\
    &-\rho_0w_t+\nabla_{u^1}\cdot(\mu_0\mathbb{D}_{u^1}w),\\
    \nabla_{u^1}\cdot z=&0,
    \end{aligned}
\end{equation*}
and thus we have that
\begin{equation}\label{zsplit}
    \begin{aligned}
    \frac12\frac{d}{dt}\|\sqrt{\rho_0}\minspace z\|_{L^2}^2+\frac12\|\sqrt{\mu_0}\minspace\mathbb{D}_{u^1}z\|_{L^2}^2=\sum_{j=1}^4 I_j,
    \end{aligned}
\end{equation}
with
\begin{equation*}
    I_1=\int_{\mathbb{R}^2}z\cdot \big(\nabla_{u^2}\cdot(\mu_0\mathbb{D}_{u^2} v^2)-\nabla_{u^1}\cdot(\mu_0\mathbb{D}_{u^1} v^2)\big) dy,
\end{equation*}
\begin{equation*}
    I_2=\int_{\mathbb{R}^2}z\cdot (\nabla_{u^1}Q^2-\nabla_{u^2}Q^2) dy,
\end{equation*}
\begin{equation*}
    I_3=-\int_{\mathbb{R}^2}\rho_0 w_t\cdot z dy,
\end{equation*}
\begin{equation*}
    I_4=\int_{\mathbb{R}^2}\nabla_{u^1}\cdot(\mu_0\mathbb{D}_{u^1}w)\cdot z dy,
\end{equation*}
where we have used that $\nabla_{u^1}\cdot z=0$.
We proceed to estimate each of these terms. Notice first that if $ v^1$, $ v^2$ satisfy \eqref{small_lips}, then from \eqref{small_lips2} we obtain that
\begin{equation}\label{A_t}
    \|\delta A(t)\|_{L^2}\leq Ct^{\frac12}\|\nabla \delta  v\|_{L^2_tL^2}. 
\end{equation}
Next, we integrate by parts to obtain that
\begin{equation*}
    \begin{aligned}
|I_1|&\leq \int_{\mathbb{R}^2}\mu_0 |\nabla z|\big|(\delta A A_2^*+A_1\delta A^*)\nabla  v^2+\delta A(\nabla  v^2)^*A_2+A_1(\nabla v^2)^*\delta A\big|dy,
    \end{aligned}
\end{equation*}
and therefore
\begin{equation*}
    |I_1|\leq C\|\nabla z\|_{L^2}t^{\frac12}\|\nabla  v^2\|_{L^\infty}t^{-\frac12}\|\delta A\|_{L^2}.
\end{equation*}
Thus, using \eqref{A_t} and integrating in time,
\begin{equation*}
    \int_0^T |I_1|dt\leq C\|\nabla \delta  v\|_{L^2_TL^2}^2 \|t^{\frac12}\nabla  v^2\|_{L^2_TL^\infty}^2+\frac{\mu^m}{8}\|\nabla z\|_{L^2_TL^2}^2.
\end{equation*}

Similarly, integration by parts and identity \eqref{iden_div} provide that
\begin{equation*}
    \begin{aligned}
    I_2=\int_{\mathbb{R}^2}\delta A^* :\nabla z \minspace Q^2dy.
    \end{aligned}
\end{equation*}
Now, we substitute the pressure by its expression in \eqref{pressure} to get
\begin{equation*}
    \begin{aligned}
    I_2=\int_{\mathbb{R}^2}\delta A^*:\nabla z (\nabla\cdot\nabla\cdot\Delta^{-1})(\mu^2\mathbb{D}u^2)(X^2(y))dy-\int_{\mathbb{R}^2}\delta A^*:\nabla z (\Delta^{-1}\nabla\cdot)(\rho^2D_t u^2)(X^2(y)) dy,
    \end{aligned}
\end{equation*}
and using again \eqref{iden_div} we integrate by parts back in the second term,
\begin{equation*}
    \begin{aligned}
    I_2=&\int_{\mathbb{R}^2}\delta A^*:\nabla z (\nabla\cdot\nabla\cdot\Delta^{-1})(\mu^2\mathbb{D}u^2)(X^2(y))dy\\
    &+\int_{\mathbb{R}^2}\delta A \minspace z\cdot(\nabla  \Delta^{-1}\nabla\cdot)(\rho^2D_t u^2)(X^2(y))\nabla X^2(y) dy.
    \end{aligned}
\end{equation*}
Therefore, 
\begin{equation*}
    \begin{aligned}
    |I_2|\leq & C\|t^{-\frac12}\delta A\|_{L^2}\|t^{\frac12}(\nabla\cdot\nabla\cdot\Delta^{-1})(\mu^2\mathbb{D}u^2)\|_{L^\infty}\|\nabla z\|_{L^2}\\
    &+C\|t^{-\frac12}\delta A\|_{L^2}\|t^{\frac12}\nabla \Delta^{-1}\nabla\cdot(\rho^2 D_t  u^2)\|_{L^6}\|z\|_{L^2}^{\frac23}\|\nabla z\|_{L^2}^{\frac13},
    \end{aligned}
\end{equation*}
where we have used \eqref{L4}. Calderon-Zygmund and Young's inequalities provide
\begin{equation*}
\begin{aligned}
\int_0^T|I_2|dt\leq& C\|\nabla \delta  v\|_{L^2_TL^2}^2\|t^{\frac12} (\nabla\cdot\nabla\cdot\Delta^{-1})(\mu^2\mathbb{D}u^2)\|_{L_T^2L^\infty}^2+\frac{\mu^m}{16} \|\nabla z\|_{L^2_TL^2}^2\\
&+C\|\nabla \delta  v\|_{L^2_T(L^2)}\|t^{\frac12}\rho^2 D_t u^2\|_{L^\frac65_TL^6}\|z\|_{L^\infty_TL^2}^{\frac23}\|\nabla z\|_{L^2_TL^2}^\frac13\\
\leq& C\|\nabla \delta  v\|_{L^2_TL^2}^2(\|t^{\frac12} (\nabla\cdot\nabla\cdot\Delta^{-1})(\mu^2\mathbb{D}u^2)\|_{L_T^2L^\infty}^2+\|t^{\frac12}\rho^2 D_t u^2\|_{L^\frac65_TL^6}^2)\\
&+\frac14\|z\|^2_{L^\infty_TL^2}+\frac{\mu^m}{8} \|\nabla z\|_{L^2_TL^2}^2.
\end{aligned}
\end{equation*}
We are left to show that $t^\frac12(\nabla\cdot\nabla\cdot\Delta^{-1})(\mu^2\mathbb{D}u^2)\in L^2(0,T;L^\infty)$ and $t^\frac12  D_tu^2\in L^\frac65(0,T;L^6)$. The latter follows from \eqref{L4} and the regularity estimates in Theorem \ref{Case1}. For the first, it is enough to define $y(T)=\int_0^T\tau \|\nabla u\|_{L^\infty}^2d\tau$ and repeat the Steps $4$ and $5$, by noticing that instead of \eqref{I8bound} now we would have 
\begin{equation*}
t\|\nabla u\|_{L^{\frac{4}{1-\gamma}}}+t\|D_t u\|_{L^{\frac{4}{1-\gamma}}}\leq c(T),    
\end{equation*}
which follows by the regularity estimates in Theorem \ref{Case1}. The constant $c(T)$ above is continuous, increasing, and such that $c(0)=0$.
We are done with $I_1$ and $I_2$. To deal with $I_3$ and $I_4$, we need to study $w$ first.
\begin{lemma}\label{w_lemma}
Let $A(t)$ be a matrix-valued function on $[0,T]\times \mathbb{R}^2$ satisfying
\begin{equation*}
    \rm{det}{A}=1.
\end{equation*}
There exists a constant $c$ such that if
\begin{equation*}
    \|\mathbb{I}_2-A\|_{L^\infty_TL^\infty}+\|A_t\|_{L^{\frac65}_TL^6}\leq c,
\end{equation*}
then for all functions $g$ in $L^2(0,T;L^2)$ satisfying
\begin{equation*}
    g=\nabla\cdot R,\quad R\in L^{\infty}_TL^2, \quad R_t\in L^{\frac65}_TL^{\frac32},
\end{equation*}
the equation 
\begin{equation*}
    \nabla \cdot(A w)=g\qquad \text{ in } [0,T]\times\mathbb{R}^2
\end{equation*}
has a solution $w$ in the space
\begin{equation*}
    W_T=\{w\in L^{\infty}_TL^2,\nabla w\in L^2_TL^2, w_t\in L^{\frac65}_TL^{\frac32}\},
\end{equation*}
that satisfies
\begin{equation*}
\begin{aligned}
    \|w\|_{L^{\infty}_TL^2}&\leq C\|R\|_{L^{\infty}_TL^2},\\
    \|\nabla w\|_{L^2_TL^2}&\leq C\|g\|_{L^2_TL^2},\\
    \|w_t\|_{L^{\frac65}_TL^{\frac32}}&\leq C\|R\|_{L^\infty_TL^2}+C\|R_t\|_{L^{\frac65}_TL^{\frac32}}.
\end{aligned}
\end{equation*}
\end{lemma}
\begin{proof}
The proof follows as in Lemma A.2. of \cite{DanchinMucha2019} with minor modifications.
\end{proof}
\begin{lemma}\label{w_lemma2}
The solution $w$ to \eqref{w_equation} given by Lemma \ref{w_lemma} satisfies
\begin{equation}\label{w_estimates}
    \begin{aligned}
    \|w\|_{L^{\infty}_TL^2}+\|\nabla w\|_{L^2_TL^2}+\|w_t\|_{L^{\frac65}_TL^{\frac32}}\leq c(T)\|\nabla \delta  v\|_{L^2_TL^2},
    \end{aligned}
\end{equation}
where $c(T)$ is a continuous increasing function of $T$ with $c(0)=0$.
\end{lemma}
\begin{proof}
Using \eqref{Lpv} and the estimates in Theorem \ref{Case1}, we have that $\nabla  v \in L^{\frac65}(0,T;L^6)$, and thus there exists a constant $c$ such that if
\begin{equation*}
    \|\nabla  v^1\|_{L^1_TL^\infty}+\|\nabla  v^1\|_{L^{\frac65}_TL^6}\leq C,
\end{equation*}
then, by Lemma \ref{w_lemma} and identity \eqref{iden_div},
\begin{equation*}
\begin{aligned}
\|w\|_{L^\infty_TL^2}&\leq C\|\delta A  v^2\|_{L^\infty_TL^2},\\
\|\nabla w\|_{L^2_TL^2}&\leq C\|\delta A^*:\nabla  v^2\|_{L^2_TL^2},\\
\|w_t\|_{L^{\frac65}_TL^{\frac32}}&\leq C\|\delta A  v^2\|_{L^\infty_TL^2}+C\|(\delta A v^2)_t\|_{L^{\frac65}_TL^\frac32}.
\end{aligned}
\end{equation*}
Using H\"older's inqueality and \eqref{A_t} repeatedly, we obtain that
\begin{equation*}
    \|\delta A  v^2\|_{L^\infty_TL^2}\leq C\|t^\frac12 v^2\|_{L^\infty_TL^\infty}\|\nabla \delta  v\|_{L^2_TL^2}\leq C\|t^\frac12 v^2\|_{L^\infty_TW^{1,\frac{2}{1-\gamma}}}\|\nabla \delta  v\|_{L^2_TL^2},
\end{equation*}
\begin{equation*}
\|\delta A^*:\nabla  v^2\|_{L^2_TL^2}\leq C\|\nabla \delta  v\|_{L^2_TL^2}\|t^{\frac12}\nabla  v^2\|_{L^2_TL^\infty},
\end{equation*}
\begin{equation*}
\|\delta A_t v^2\|_{L^{\frac65}_TL^\frac32}\leq C\|\delta A_t\|_{L^2_TL^2}\| v^2\|_{L^3_TL^6}\leq C\|\nabla \delta v\|_{L^2_TL^2}\| v^2\|_{L^3_TL^6},
\end{equation*}
\begin{equation*}
    \|\delta A v^2_t\|_{L^{\frac65}_TL^\frac32}\leq \|\nabla\delta  v^2\|_{L^2_TL^2}\|t^\frac12  v^2_t\|_{L^\frac{6}{5}_TL^{6}},
\end{equation*}
so, by the regularity provided in Theorem \ref{Case1}, the proof is concluded.
\end{proof}
Now, we go back to estimate the terms $I_3$, $I_4$ in \eqref{zsplit}. H\"older's inequality and Lemma \ref{w_lemma2} provides that
\begin{equation*}
\begin{aligned}
    \int_0^T|I_3|dt&\leq C\|w_t\|_{L^{\frac65}_TL^\frac32}\|z\|_{L^6_TL^3} \leq c(T)\|\nabla \delta  v\|_{L^2_TL^2}\|z\|_{L^\infty_TL^2}^{\frac23}\|\nabla z\|_{L^2_TL^2}^{\frac13}\\
    &\leq c(T)\|\nabla \delta  v\|_{L^2_TL^2}^2+\frac14\|z\|_{L^\infty_TL^2}^{2}+\frac{\mu^m}{8}\|\nabla z\|_{L^2_TL^2}^{2}
\end{aligned}
\end{equation*}
and 
\begin{equation*}
    \int_0^T|I_4|dt\leq \frac{\mu^m}8\|\nabla z\|_{L^2_TL^2}^2+c(T)\|\nabla \delta  v\|_{L^2_TL^2}^2.
\end{equation*}
Hence, joining the estimates for $I_1$-$I_4$ and going back to \eqref{zsplit}, we obtain that for small $T>0$, 
\begin{equation}\label{auxxx}
    \sup_{t\in[0,T]}\|z\|_{L^2}^2+\int_0^T\|\nabla z \|_{L^2}^2dt\leq c(T)\int_0^T\|\nabla \delta v\|_{L^2}^2dt.
\end{equation}
Recalling that $\delta v=w+z$ and the estimate for $w$ \eqref{w_estimates}, we obtain that
\begin{equation*}
    \int_0^T\|\nabla \delta v\|_{L^2}^2dt\leq c(T)\int_0^T\|\nabla \delta v\|_{L^2}^2dt,
\end{equation*}
and hence we conclude \eqref{aux_uniq}, i.e., that for $T>0$ small enough $\|\nabla \delta v\|_{L^2_TL^2}=0$. Plugging this back into \eqref{auxxx} and \eqref{w_estimates} allow us to conclude that
\begin{equation*}
    v^1\equiv v^2 \quad \text{ on } [0,T]\times \mathbb{R}^2.
\end{equation*}
One can now go back to Eulerian coordinates, while the passage to arbitrary $T>0$ follows from standard connectivity arguments.
\qed

\vspace{0.5cm}

\subsection*{Acknowledgements}
The authors would like to gratefully thank the referees for their careful reading of
the manuscript.
FG and EGJ were partially supported by the ERC through the Starting Grant project H2020-EU.1.1.-639227. FG was partially supported by the grant EUR2020-112271 (Spain). EGJ was partially supported by the ERC Starting Grant ERC-StG-CAPA-852741. This project has received funding from the European Union’s Horizon 2020 research and innovation programme under the Marie Skłodowska-Curie grant agreement CAMINFLOW No 101031111.

\bibliographystyle{plain}
\bibliography{references9}

\end{document}